\documentclass[11pt]{amsart}
\usepackage{amssymb}
\usepackage{amssymb, amsmath}
\usepackage{verbatim}
\let\pa=\partial

\let\eps=\varepsilon

\let\s=\sigma
\let\f=\frac

\let\wt=\widetilde

\def\cA{{\mathcal A}}

\def\cC{{\mathcal C}}

\def\cO{{\mathcal O}}
\def\cP{{\mathcal P}}
\def\cQ{{\mathcal Q}}
\def\cR{{\mathcal R}}
\def\cS{{\mathcal S}}
\def\cT{{\mathcal T}}

\def\cV{{\mathcal V}}

\def\bH{{\mathbb H}}

\def\pa{\partial}

\def\Supp{\, \mbox{Supp}\,  }

\def\virgp{\raise 2pt\hbox{,}}
\def\cdotpv{\raise 2pt\hbox{;}}

\def\Id{\mathop{\rm Id}\nolimits}

\def\C{\mathop{\mathbb C\kern 0pt}\nolimits}
\def\DD{\mathop{\mathbb D\kern 0pt}\nolimits}
\def\EE{\mathop{{\mathbb E \kern 0pt}}\nolimits}
\def\K{\mathop{\mathbb K\kern 0pt}\nolimits}
\def\N{\mathop{\mathbb N\kern 0pt}\nolimits}
\def\Q{\mathop{\mathbb Q\kern 0pt}\nolimits}
\def\R{\mathop{\mathbb R\kern 0pt}\nolimits}
\def\SS{\mathop{\mathbb S\kern 0pt}\nolimits}
\def\ZZ{\mathop{\mathbb Z\kern 0pt}\nolimits}
\def\TT{\mathop{\mathbb T\kern 0pt}\nolimits}
\def\P{\mathop{\mathbb P\kern 0pt}\nolimits}
\def\H{\mathop{\mathbb H\kern 0pt}\nolimits}
\def\div{ \hbox{\rm div}\,  }
\def\curl{ \hbox{\rm curl}\,  }
\def\Ker{\hbox{\rm Ker}\,}

\newcommand{\dv}{{\delta\! v}^\eps}
\newcommand{\dT}{{\delta\!\Theta}^\eps}

\newcommand{\dV}{{\delta\!V}^\eps}

\newcommand{\Int}{\displaystyle \int}

\newcommand{\Sum}{\displaystyle \sum}

\def\ddj{\dot \Delta_j}

\newcommand{\Z}{{\ZZ}}

\def\na{\nabla}

\newcommand{\ef}{ \hfill $ \blacksquare $ \vskip 3mm}

\newcommand{\beq}{\begin{equation}}
\newcommand{\eeq}{\end{equation}}
\newcommand{\ben}{\begin{eqnarray}}
\newcommand{\een}{\end{eqnarray}}
\newcommand{\beno}{\begin{eqnarray*}}
\newcommand{\eeno}{\end{eqnarray*}}

\newtheorem{defi}{Definition}[section]
\newtheorem{thm}{Theorem}[section]
\newtheorem{lem}{Lemma}[section]
\newtheorem{rmk}{Remark}[section]

\newtheorem{prop}{Proposition}[section]

\begin{document}
\title[The  Oberbeck-Boussinesq approximation in critical spaces]
{The  Oberbeck-Boussinesq approximation in critical spaces}
\author[R. Danchin]{Rapha\"{e}l Danchin}
\address[R. Danchin]{Universit\'{e} Paris-Est,  LAMA (UMR 8050), UPEMLV, UPEC, CNRS, Institut Universitaire de France,
 61 avenue du G\'{e}n\'{e}ral de Gaulle, 94010 Cr\'{e}teil Cedex, France.} \email{raphael.danchin@u-pec.fr}
\author[L. He]{Lingbing He}
\address[L. He]{Department of Mathematical Sciences, Tsinghua University\\
Beijing 100084,  P. R.  China.} \email{lbhe@math.tsinghua.edu.cn}

\begin{abstract} In this paper we study the  validity of the so-called
Oberbeck-Boussinesq approximation for compressible viscous perfect gases  in the whole
three-dimensional space. Both the cases of fluids with positive heat conductivity and zero conductivity
are considered.  For  small perturbations of  a constant equilibrium,
we establish the global existence of unique strong solutions   in a critical regularity functional framework.
Next, taking advantage of Strichartz estimates for the associated system of acoustic
waves, and of uniform estimates with respect to the Mach number, we obtain
all-time convergence to the Boussinesq system with a explicit decay rate.
\end{abstract}
\maketitle

\setcounter{equation}{0}

\section{Introduction}

This work aims at giving  a mathematical justification of the \emph{Oberbeck-Boussinesq approximation}
that is commonly used to model stratified fluids such as e.g. atmosphere or oceans.
One of the characteristics of the this approximation is that, although the primitive system
is the full compressible Navier-Stokes system,
 the limit equations are incompressible, and the density is a  constant.  In fact, the velocity field
just convects an active scalar creating buoyancy force, proportional to the discrepancy between the temperature and
its equilibrium.

\subsection{Formal derivation}
The starting point of our analysis is the full Navier-Stokes system for compressible viscous fluids, namely
$$
  \left\{
    \begin{aligned}
      & \pa_t \rho+\div (\rho u)=0,\\
      &\pa_t(\rho u)+\div(\rho u\otimes u)-\div\tau
      +\frac{1}{{\rm Ma}^2}\nabla P=\frac1{{\rm Fr}^2}\,\rho\nabla V,\\
      &\pa_t(\rho s)+\div(\rho s)+\div(q/\mathcal{T})=\sigma.
    \end{aligned}\right.
  $$
Above
 $\rho=\rho(t,x)\in\R^+,$  $u=u(t,x)\in\R^3$ and
  $\mathcal{T}=\mathcal{T}(t,x)\in\R^+$
  stand  for the density, velocity field and  temperature, respectively.
   The scalar function $V$ stands for some (given) external potential  (e.g. the
 gravity potential).
 We concentrate on the study of the evolution toward the future
 in the whole space $\R^3$ (hence the time variable $t$ belongs to $\R^+$
 and the space variable $x,$ to $\R^3$).

  In the Newtonian case that we shall consider,
  the stress tensor $\tau$ is given by
  $$
  \tau=\mu(\nabla u+Du)+\lambda \div u\:{\rm Id}.
  $$
 For simplicity, the viscosity coefficients $\lambda$ and $\mu$ are assumed to be constant.
 As we only consider \emph{viscous} fluids, those two coefficients satisfy
   $$\mu>0\quad\hbox{and}\quad\nu:= \lambda+2\mu>0.$$
  This  ensures ellipticity for the second order operator $\mathcal{A}:= \mu\Delta+(\lambda+\mu)\na\div$.
\smallbreak \noindent
 The heat flux $q$ is equal to $-\kappa\nabla\cT$ for some  constant   conductivity coefficient $\kappa\geq0.$
 The pressure $P,$ the internal energy $e$ and the specific entropy $s$
 are   related to $\rho$ and $\cT$ through the Gibbs relation
  $${\mathcal T}\,ds=de+P\,d(1/\rho).$$

  We focus on perfect gases, namely we assume  that  for some $a>0$ and $b>0,$
   $P= a\rho \mathcal{T}$ and $e=b{\mathcal T}.$
   After rescaling, it is non restrictive to  take $a=b=1.$

  Finally, in the velocity equation, the Mach number  ${\rm Ma}$ and the Froude number ${\rm Fr}$ are two
  dimensionless small parameters
  accounting for the compressibility and the stratification of the fluid.
Formally,
Oberbeck-Boussinesq approximation is obtained in the  asymptotics $\eps\to0$ if
$$
{\rm Ma}=\eps\quad\hbox{and}\quad{\rm Fr}=\sqrt\eps,
$$
an assumption that we shall make from now on.
\medbreak
Gathering  all the above assumptions over the coefficients and state laws,
we end up with the following system (with exponents $\eps$ emphasizing the dependency
with respect to $\eps$):
\begin{equation}
  \left\{
    \begin{aligned}
      & \pa_t \rho^\eps+\div (\rho^\eps u^\eps)=0,\\
      &\pa_t(\rho^\eps u^\eps)+\div(\rho^\eps u^\eps\otimes u^\eps)-\mu\Delta u^\eps-(\lambda+\mu)\na\div u^\eps +\f{\na P^\eps}{\eps^2}=\f1\eps\rho^\eps\na V^\eps,\\
      &\pa_t(\rho^\eps \mathcal{T}^\eps)+\div(u^\eps \rho^\eps\mathcal{T}^\eps)-\kappa\Delta \mathcal{T}^\eps=\eps^2[2\mu|Du^\eps|^2+\lambda (\div u^\eps)^2].
    \end{aligned}
  \right.
  \label{eq:oNS}
\end{equation}

 Let us first provide a formal derivation of the Oberbeck-Boussinesq approximation
in the case where the heat conductivity $\kappa$ is positive.
We want to consider   so-called  \emph{ill-prepared data} of the form
   $\rho_0^\eps=1+\eps a_0^\eps,$ $u_0^\eps$  and $\mathcal{T}^\eps_0=1+\eps \theta_0^\eps$
   where $(a_0^\eps,u_0^\eps,\theta_0^\eps)$ are bounded in a sense that will be specified later on.
 Setting $\rho^\eps =1+\eps a^\eps$ and $\mathcal{T}^\eps=1+\eps \theta^\eps,$
we get the following governing equations for $(a^\eps,u^\eps,\theta^\eps)$:
\begin{equation}
  \left\{
    \begin{aligned}
      & \pa_t a^\eps+\f{\div u^\eps}{\eps}=-\div (a^\eps u^\eps),\\
      &\pa_t u^\eps + u^\eps\cdot\na u^\eps-\f{\mathcal{A}u^\eps}{1+\eps a^\eps} +\f{\na (a^\eps+\theta^\eps+\eps a^\eps\theta^\eps)}{\eps(1+\eps a^\eps)}=\f1\eps\na V^\eps,\\
      &\pa_t \theta^\eps  +\f{\div u^\eps}{\eps}+\div( \theta^\eps u^\eps) -\f{\kappa\Delta \theta^\eps}{1+\eps a^\eps}=\frac{\eps}{1+\eps a^\eps}[2\mu|Du^\eps|^2+\lambda (\div u^\eps)^2].
    \end{aligned}
  \right.
  \label{eq:hNS}
\end{equation}

In order to handle the singular potential term in the r.h.s. of the velocity equation, it is usual to
work with the \emph{modified} deviation of density $b^\eps:=a^\eps-V^\eps.$
We get
\begin{equation}  \label{eq:hNSb}
  \left\{
    \begin{aligned}
      & \pa_t b^\eps+u^\eps\cdot\nabla b^\eps+\f{\div u^\eps}{\eps}=-\pa_tV^\eps-\div(V^\eps u^\eps)-b^\eps \div u^\eps,\\
      &\pa_t u^\eps + u^\eps\cdot\na u^\eps-\cA u^\eps +\f{\nabla(b^\eps+\theta^\eps)}\eps
      =\biggl(\frac{a^\eps-\theta^\eps}{1+\eps a^\eps}\biggr)\nabla a^\eps-
      \f{\eps a^\eps}{1+\eps a^\eps} \cA u^\eps,\\
      &\pa_t \theta^\eps  +u^\eps\cdot\nabla\theta^\eps+\f{\div u^\eps}{\eps}
      -\kappa\Delta\theta^\eps=\frac{\eps}{1+\eps a^\eps}[2\mu|Du^\eps|^2+\lambda (\div u^\eps)^2]
      \\[-1ex]&\hspace{9cm}-\kappa\f{\eps a^\eps}{1+\eps a^\eps}\Delta \theta^\eps-\theta^\eps\div u^\eps,
    \end{aligned}
  \right.
\end{equation}
which may formally written as follows:
$$
\frac{\pa}{\pa t}\left(\begin{array}{c}b^\eps\\ u^\eps\\\theta^\eps\end{array}\right)
+\frac1\eps\left(\begin{array}{ccc}0&\div&0\\\nabla&0&\nabla\\0&\div&0\end{array}\right)
\left(\begin{array}{c}b^\eps\\ u^\eps\\\theta^\eps\end{array}\right)=\cO(1).
$$
The notation  $\cO(1)$  designates  terms that are expected to be bounded uniformly with respect to $\eps.$
\medbreak
As a consequence of our considering ill-prepared data,  the first order time derivatives
are likely to blow-up like $1/\eps$ for $\eps$ going to $0.$ At the `physical' level, this means that
 highly oscillating acoustic waves may propagate in the fluid.

 In order to better understand  the action of those singular terms,
 we may first look at the kernel $\Ker L$ of
 the $5\times 5$ first order  antisymmetric differential matrix operator $L$ above.
 The basic idea is that modes that are
in $\Ker L$ will not be affected, while modes that are in $(\Ker L)^\perp$ may experience  wild oscillations.
A straightforward computation shows that
$$\begin{array}{lll}
\Ker L&=&\Bigl\{(b,u,\theta): \div u=0\quad\hbox{and}\quad \nabla(b+\theta)=0\Bigr\},\\[1ex]
(\Ker L)^\perp&=&\Bigl\{(b,u,\theta): \curl u=0\quad\hbox{and}\quad \nabla(b-\theta)=0\Bigr\}\cdotp
\end{array}
$$
Hence it is natural to look more closely at the equations satisfied by
$(q^\eps,\cQ u^\eps)$ and $(\Theta^\eps,\cP u^\eps)$
where $\cP$ and $\cQ$ stand for the orthogonal projectors over divergence-free
and curl-free vector fields, respectively,  and
$$q^\eps:=\frac{\theta^\eps+b^\eps}{\sqrt2},\quad
\Theta^\eps:= \frac{\theta^\eps-b^\eps}{\sqrt2}\cdotp
$$
As $L$ is  antisymmetric, we expect
the oscillating components of the solution, namely $\cQ u^\eps$ and $q^\eps$ to be dispersed
whereas $L$ will have no effect on $\cP u^\eps$  and $\Theta^\eps.$ Let us be more accurate:
we see that $(q^\eps,\cQ u^\eps)$ satisfies
\begin{equation}\label{eq:osc}
  \left\{
    \begin{aligned}
      & \pa_t q^\eps +\f{\sqrt2}\eps{\div \cQ u^\eps}=-\div(q^\eps u^\eps)
      -\frac{\sqrt2}2\biggl(\pa_tV^\eps+\div(V^\eps u^\eps)+\kappa\frac{\Delta\theta^\eps}{1+\eps a^\eps}\biggr)\\
      &\hspace{7cm}+\frac{\sqrt2}2\frac{\eps}{1+\eps a^\eps}[2\mu|Du^\eps|^2+\lambda (\div u^\eps)^2],\\
      &\pa_t \cQ u^\eps + \f{\sqrt2}\eps\nabla q^\eps=\cQ\biggl(\biggl(\frac{a^\eps-\theta^\eps}{1+\eps a^\eps}\biggr)\nabla a^\eps-
      \f{\cA u^\eps}{1+\eps a^\eps} -u^\eps\cdot\nabla u^\eps\biggr)
    \end{aligned}
  \right.
\end{equation}
whereas $(\Theta^\eps,\cP u^\eps)$ fulfills
\begin{equation}\label{eq:nonosc}
 \! \left\{
    \begin{aligned}
      & \!\pa_t \Theta^\eps +\cP u^\eps\!\cdot\!\nabla\Theta^\eps-\frac\kappa2\Delta\Theta^\eps
      =-\div(\Theta^\eps\cQ u^\eps)\!+\!\frac{\sqrt2}2\bigl(\pa_tV^\eps+\cP u^\eps\!\cdot\!\nabla V^\eps\!+\!\div(V^\eps\cQ u^\eps)\bigr)
      \\&\hspace{2.8cm}+\f\kappa2\Delta q^\eps-\frac{\sqrt2\kappa}2\frac{\eps a^\eps}{1+\eps a^\eps}\Delta\theta^\eps
      +\frac{\sqrt2}2\frac{\eps}{1+\eps a^\eps}[2\mu|Du^\eps|^2+\lambda (\div u^\eps)^2],\\
      &\!\pa_t \cP u^\eps-\mu\Delta\cP u^\eps +\cP(\cP u^\eps\cdot\nabla\cP u^\eps)
      +\cP(\theta^\eps\nabla a^\eps)=-\cP\bigl( u^\eps\cdot\nabla\cQ u^\eps+\cQ u^\eps\cdot\nabla\cP u^\eps\bigr) \\&
     \hspace{6.3cm} -\cP\biggl(\frac{\eps a^\eps}{1+\eps a^\eps}\cA u^\eps\biggr)
      +\cP\biggl(\frac{\eps a^\eps(\theta^\eps-a^\eps)}{1+\eps a^\eps}\nabla a^\eps\biggr)\cdotp
          \end{aligned}
  \right.
\end{equation}

If we assume  the solution $(b^\eps,u^\eps,\theta^\eps)$ and the data to be  bounded independently of $\eps$
then the right-hand side of \eqref{eq:osc} is bounded, too.
Hence, owing to the antisymmetric (and nondegenerate) structure of the left-hand side of \eqref{eq:osc},
one may expect $(q^\eps,\cQ u^\eps)$ to tend weakly to $0.$
We shall see later on that in the whole space setting that is here considered, it is possible to get
strong convergence (for suitable negative Besov norms), with an explicit rate.
\smallbreak

In order to find out what the limit system  for \eqref{eq:nonosc} is, let
us observe that
\begin{align}\label{eq:relation}
\sqrt2\cP(\theta^\eps\nabla a^\eps)
&=\cP(\Theta^\eps\nabla V^\eps)+\cP(q^\eps\nabla V^\eps)
+\sqrt2\cP(\theta^\eps\nabla b^\eps)\nonumber\\
&=\cP(\Theta^\eps\nabla V^\eps)+\cP(q^\eps\nabla V^\eps)
+\sqrt2\cP((\theta^\eps+b^\eps)\nabla b^\eps)\nonumber\\
&=\cP(\Theta^\eps\nabla V^\eps)+\cP(q^\eps\nabla V^\eps)
+2\cP(q^\eps\nabla b^\eps).
\end{align}
Because $q^\eps$  tends  to $0,$ we expect that
$$
\sqrt2\cP(\theta^\eps\nabla a^\eps)-\cP(\Theta^\eps\nabla V^\eps)\rightarrow0\quad\hbox{for }\  \eps\ \hbox{ going to }\ 0.
$$
Hence, if we assume in addition that $V^\eps\to V,$ $\cP u^\eps_0\to v_0$ and $\Theta_0^\eps\to\Theta_0,$
then  $(\Theta^\eps,\cP u^\eps)$ should tend to the solution $(\Theta,v)$ to the following \emph{Boussinesq system:}
\begin{equation}
  \left\{
    \begin{aligned}
     &\pa_t \Theta+v\cdot \na \Theta-\f{\kappa}2 \Delta \Theta=\f{\sqrt2}2 (\pa_t+v\cdot\nabla)V,\\
          &\pa_t v+v\cdot\na v-\mu \Delta v+\na \Pi = -\f{\sqrt2}2\Theta\nabla V,\qquad\div v=0,\\
          &(\Theta,v)|_{t=0}=(\Theta_0,v_0).
    \end{aligned}
  \right.
  \label{eq:BouS}
  \end{equation}
  Setting $\tilde\Theta=\Theta-\sqrt2/2\:V,$ and changing $\nabla\Pi$ accordingly, we see that this system is equivalent to
  the following one, which is commonly used:
  \begin{equation}
  \left\{
    \begin{aligned}
     &\pa_t \tilde\Theta+v\cdot \na \tilde\Theta-\f{\kappa}2 \Delta \tilde\Theta=\f{\sqrt2}4\kappa\Delta V,\\
          &\pa_t v+v\cdot\na v-\mu \Delta v+\na\tilde\Pi = -\f{\sqrt2}2\tilde\Theta\nabla V,\qquad\div v=0.
    \end{aligned}
  \right.
  \label{eq:BouSprime}
  \end{equation}
   Note that  although the density is constant in the limit system, it comes into play in the buoyancy force where it
is related to the temperature and the potential.
\medbreak
We end this paragraph with a formal derivation \emph{in  the case  $\kappa=0.$}
It turns out to be easier to  work with the pressure rather than with
the temperature.
We thus set $\rho^\eps =1+\eps a^\eps$ and $P^\eps=\rho^\eps\mathcal{T}^\eps=1+ \eps (\mathcal{R}^\eps+V^\eps),$
and obtain that
\begin{equation}
  \left\{
    \begin{aligned}
      & \pa_t a^\eps+\f{\div u^\eps}{\eps}=-\div (a^\eps u^\eps),\\
      &\pa_t u^\eps + u^\eps\cdot\na u^\eps-\f{\mathcal{A}u^\eps}{1+\eps a^\eps} +\f{\na \mathcal{R}^\eps}{\eps(1+\eps a^\eps)}=\frac{a^\eps}{1+\eps a^\eps}\nabla V^\eps,\\
      &\pa_t \mathcal{R}^\eps  +\f{\div u^\eps}{\eps}+\div( \mathcal{R}^\eps u^\eps) = \eps [2\mu|Du^\eps|^2+\lambda (\div u^\eps)^2]
      -\pa_tV^\eps-\div(V^\eps u^\eps).
    \end{aligned}
  \right.
  \label{eq:nhNS}
  \end{equation}
  Setting $\Theta^\eps:=a^\eps-\cR^\eps-V^\eps,$ we thus get
   \begin{equation}
  \left\{
    \begin{aligned}
      & \!\pa_t \Theta^\eps+\div (\Theta^\eps u^\eps)=-\eps  [2\mu|Du^\eps|^2+\lambda (\div u^\eps)^2],\\
       &\!\pa_t\cP u^\eps \!+\! \cP(u^\eps\cdot\na u^\eps)-\mu\Delta\cP u^\eps=-\cP\biggl(\f{\eps a^\eps}{1+\eps a^\eps}{\mathcal{A}u^\eps}\biggr)
       +\cP\biggl(\f{a^\eps}{1+\eps a^\eps}\nabla(V^\eps\!+\!\cR^\eps)\biggr),\\
            &\!\pa_t\cQ u^\eps \!+\! \cQ(u^\eps\!\cdot\!\na u^\eps)-\nu\Delta\cQ u^\eps\!+\!\f{\na \mathcal{R}^\eps}\eps
            =-\cQ\biggl(\!\f{\eps a^\eps}{1\!+\!\eps a^\eps}{\mathcal{A}u^\eps}\!\biggr)
            \!+\!\cQ\biggl(\!\f{a^\eps}{1\!+\!\eps a^\eps}\nabla(V^\eps\!+\!\cR^\eps)\!\biggr),\\
      &\!\pa_t \mathcal{R}^\eps  +\f{\div u^\eps}{\eps}+\div( \mathcal{R}^\eps u^\eps) = \eps [2\mu|Du^\eps|^2+\lambda (\div u^\eps)^2]
      -\pa_tV^\eps-\div(V^\eps u^\eps).
    \end{aligned}
  \right.\!\!\!
  \label{eq:nhNSbis}
  \end{equation}

As before, owing to the first order antisymmetric terms, we expect $(\cQ u^\eps,\cR^\eps)$ to go to $0.$
Concerning $(\Theta^\eps,\cP u^\eps),$ we notice that
$$
 \cP(a^\eps\nabla (V^\eps+\cR^\eps))=\cP(\Theta^\eps\nabla V^\eps)+\cP(\Theta^\eps\nabla\cR^\eps).
$$
Therefore the limit system for $(\Theta^\eps,\cP u^\eps)$ reads
\begin{equation}
  \left\{
    \begin{aligned}
         &\pa_t \Theta+v\cdot \na \Theta =0,\\
     &\pa_t v+v\cdot\na v-\mu \Delta v+\na  \Pi =  \Theta\na V,\\
     &\div v=0.
    \end{aligned}
  \right.
  \label{eq:nBouS}
  \end{equation}
  Note that in contrast with \eqref{eq:BouS}, this system is not fully parabolic.


  \subsection{Some related works}

There is an important literature dedicated to the limit system, that is the Oberbeck-Boussinesq equations
\eqref{eq:BouS}, \eqref{eq:BouSprime} and  \eqref{eq:nBouS},
under various hypotheses over the coefficients $\kappa$ and $\mu,$ and the potential $V$
(although  the most common assumption is that $V=x_3$).
Loosely speaking the classical results concerning the existence issue are
(see e.g. \cite{DP1,DP2,He} and the references therein):
  \begin{itemize}
  \item Dimension $2$: Global existence of strong solutions if $(\mu,\kappa)\not=(0,0).$
  \item Dimension $3$ with $\mu\not=0$:  Global weak solutions and
   local strong solutions (which become global if the data are small).
   \item Dimension $3$ with $\mu=0$ : only local-in-time strong solutions are available.
  \end{itemize}

In contrast, although the Oberbeck-Boussinesq approximation is commonly
used in geophysics (see e.g. the books by J. Pedlosky \cite{pedlosky}  or R. K. Zeytounian \cite{Z})
there are  few results concerning the rigorous justification of the
derivation that we presented in the previous subsection.
  To our knowledge, the first  mathematical  justification of Oberbeck-Boussinesq approximation
  in this context has been given only rather recently in the framework of
  the so-called \emph{variational weak solutions} (see \cite{F} for a complete presentation of such solutions
  for the full Navier-Stokes equations).
  The case of  bounded domains with  potential
   $V=x_3$ (or more generally, in $W^{1,\infty}(\Omega)$) has been treated
  by E. Feireisl and A. Novotny in \cite{FN-book,FN-paper}, while
  the exterior domain case has been studied by  E. Feireisl and M. Schonbek  in \cite{FS}
  (still under the assumption  $V\in W^{1,\infty}(\Omega)$,  thus ruling out the common but not so physical assumption
  that $V=x_3$).
For passing to the limit,   all those works borrow some seminal ideas that have  been introduced by P.-L. Lions
  in his book
  \cite{Lions} and B. Desjardins et al  in \cite{DGLM}
in the related context of  low Mach number limit for the isentropic Navier-Stokes
equations\footnote{For other recent results concerning the low Mach number asymptotics for the \emph{full} Navier-Stokes equations, the reader may refer to \cite{Al,HL,Hoff,Klein}.}.
\medbreak
On the one hand, those results are very general for one may consider any finite energy data.
On the other hand, the convergence results are not very accurate for they strongly rely on compactness methods :
in particular convergence holds up to extraction only,  and  no rate  may be given.


\subsection{Aim of the paper}

Getting stronger results of convergence that is in particular convergence of the \emph{whole} sequence with an explicit rate,
is the main purpose of the present work.
Considering general variational solutions is hopeless. We shall
focus on strong solutions with the so-called critical regularity, a framework
which is nowadays classical for the study of viscous compressible fluids
(see e.g. \cite{BCD,D1,D2}).
Of course, this will enforce us to restrict considerably the set of admissible data, but
we will get much more accurate results of convergence.

Working in a functional framework that has the same scaling invariance as \eqref{eq:hNS}, if any, is the basic idea.
Here we see that (if $V^\eps\equiv0$ to simplify the presentation), the system
is ``almost'' invariant for all $\ell>0$ by the rescaling
$$
a^\eps(t,x)\to a^\eps(\lambda^2t,\lambda x),\quad
u^\eps(t,x)\to \lambda u^\eps(\lambda^2t,\lambda x),\quad
\theta^\eps(t,x)\to \lambda^2\theta^\eps(\lambda^2t,\lambda x).
$$
If we believe in an energy type method then
a good candidate for initial data is thus the homogeneous Sobolev space
$$
\dot H^{\f32}(\R^3)\times\bigl(\dot H^{\f12}(\R^3)\bigr)^3\times\dot H^{-\f12}(\R^3),
$$
or rather the slightly smaller following homogeneous Besov space:
$$
\dot B^{\frac32}_{2,1}(\R^3)\times\bigl(\dot B^{\frac12}_{2,1}(\R^3)\bigr)^3
\times\dot B^{-\frac12}_{2,1}(\R^3)
$$
which has nicer embedding properties ($\dot B^{\f32}_{2,1}$ is embedded in bounded functions for instance)
and better behaves with respect to maximal parabolic estimates.\medbreak
\smallbreak
However, owing to the lower order pressure term,  the above scaling invariance is not quite respected.
Consequently,  we  have to work at a different level of
regularity for the low frequencies  of  $a^\eps$ and $\theta^\eps,$
to compensate this scaling defect.
All this is now well understood and already occurs in
the isentropic case \cite{D1}.

Finally, in the case $\kappa=0$ that we shall also consider (and that cannot be studied in the framework of
variational solutions), only the velocity is smoothed out during the evolution, and it is no longer
possible to use a critical regularity framework: we will have to assume much more regularity.
\bigbreak
We end this introductory part with a short description of the rest of the paper.
After an unavoidable introduction of some notations and functional spaces,
the next section is  devoted to the presentation of the main results of the paper.
The analysis of the heat conducting case is carried out in Section \ref{s:kappa}
while $\kappa=0$ is considered in Section \ref{s:zero}.
Some technical estimates are postponed in the Appendix.


\section{Results}\label{s:results}

Before presenting the main statements of the paper, we briefly introduce some notations and  function spaces.
We are given an homogeneous Littlewood-Paley decomposition
$(\ddj)_{j\in\Z}$ that is a dyadic decomposition in the Fourier space for $\R^3.$
One may for instance set $\ddj:=\varphi(2^{-j}D)$ with $\varphi(\xi):=\chi(\xi/2)-\chi(\xi),$
and $\chi$ a non-increasing nonnegative smooth function supported in $B(0,4/3),$ and value $1$ on $B(0,3/4)$
 (see \cite{BCD}, Chap. 2 for more details).

We then define, for $p\in[1,+\infty]$ and $s\in\R,$ the semi-norms
$$
\|z\|_{\dot B^s_{p,1}}:=\sum_{j\in\Z}2^{js}\|\ddj z\|_{L^p}.
$$

In order to avoid complications due to polynomials, we adopt the following definition of homogeneous
Besov spaces:
$$
\dot B^s_{p,1}=\Bigl\{z\in\cS'(\R^3) : \|z\|_{\dot B^s_{p,1}}<\infty \ \hbox{ and }\ \lim_{j\to-\infty} \dot S_jz=0\Bigr\}
\quad\hbox{with }\ \dot S_j:=\chi(2^{-j}D).$$

To compensate the lack of strict scaling invariance of the system under consideration (as pointed
out in the previous section), we  also need to  introduce the following  \emph{hybrid Besov spaces}
with different regularity exponent  in low and high frequencies:
\begin{defi} For  $s\in \R,$ $p\in[1,\infty]$  and $\alpha>0,$ we set
$$ \|z\|_{\tilde{B}^{s,\pm}_{p,\alpha}}:= \sum_{j\in\Z} 2^{js}\bigl(\min(\alpha^{-1},2 ^j)\bigr)^{\pm1}
\|\ddj z\|_{L^p}
$$
and  define
$$
\tilde{B}^{s,\pm}_{p,\alpha}:=\Bigl\{z\in \mathcal{S}'(\R^3) :\|z\|_{\tilde{B}^{s,\pm}_{p,\alpha}}<\infty\} \ \hbox{ and }\ \lim_{j\to-\infty} \dot S_jz=0\Bigr\}\cdotp
$$
\end{defi}
We shall mainly use the above definition with $p=2,$ in which case, the corresponding hybrid Besov space
will be simply denoted by $\tilde B^{s,\pm}_\alpha,$ if the fact that $p=2$ is clear from the context.
\medbreak
We agree that\footnote{We omit the dependency with respect to the threshold $\alpha$ in the above notation
because the value of $\alpha$ will be always clear from the context.}:
\begin{equation}
\label{eq:decompo}  z^\ell:=\sum_{2^j\alpha\leq 1} \ddj z\quad\mbox{and}\quad z^h:=\sum_{2^j\alpha>1} \ddj z. \end{equation}
With this notation, we have
$$
 \|z\|_{\tilde{B}^{s,\pm}_{p,\alpha}}=\|z^\ell\|_{\dot B^{s\pm1}_{p,1}}+\alpha^{\mp1}\|z^h\|_{\dot B^s_{p,1}}.
$$
Therefore $\dot B^s_{p,1}$ is the bulk regularity of a function in $\tilde B^{s,\pm}_{p,\alpha}$  while
the  behavior at infinity is given by the low frequency part which is in $\dot B^{s\pm1}_{p,1}.$
Of course, changing the value of $\alpha$ does not affect the space, and the corresponding norms
are equivalent. However a suitable choice of  $\alpha$  will enable us to get uniform estimates with respect to $\eps.$

As we shall work with \emph{time-dependent functions} with values in Besov spaces,
we  introduce the norms:
$$
\|u\|_{L^q_T(\dot B^s_{p,1})}:=\bigl\| \|u(t,\cdot)\|_{\dot B^s_{p,1}}\bigr\|_{L^q(0,T)}
\quad\hbox{and}\quad
\|u\|_{L^q_T(\tilde B^{s,\pm1}_{p,\alpha})}:=\bigl\| \|u(t,\cdot)\|_{\tilde B^{s,\pm1}_{p,\alpha}}\bigr\|_{L^q(0,T)}.
$$
As in many works using parabolic estimates in Besov spaces, it is somehow natural
to take the time-Lebesgue norm \emph{before} performing the summation for computing
the Besov norm.  This motivates us to introduce the following quantities:
$$
 \|u\|_{\tilde L_T^q(\dot {B}^{s}_{p,1})}:= \sum_{j\in\Z} 2^{js}
\|\ddj u\|_{L_T^q(L^p)}\!\!\quad\hbox{and}\quad\!\!
 \|u\|_{\tilde L_T^q(\tilde{B}^{s,\pm}_{p,\alpha})}:= \sum_{j\in\Z} 2^{js}\bigl(\min(\alpha^{-1},2 ^j)\bigr)^{\pm1}
\|\ddj u\|_{L_T^q(L^p)}.
$$
The index $T$ will be omitted if $T=+\infty$ and we shall denote by $\tilde C(\dot B^s_{p,1})$
(resp. $\tilde C(\tilde B^{s,\pm}_{p,\alpha})$) the subset of $\tilde L^\infty(\dot B^s_{p,1})$
(resp. $\tilde L^\infty(\tilde B^{s,\pm}_{p,\alpha})$) constituted by continuous functions over $\R^+$ with values in $\dot B^s_{p,1}$
(resp. $\tilde B^{s,\pm}_{p,\alpha}$).
\smallbreak
Let us emphasize that, owing to Minkowski inequality, we have
$$
 \|u\|_{L_T^q(\dot {B}^{s}_{p,1})}\leq \|u\|_{\tilde L_T^q(\dot {B}^{s}_{p,1})}
$$
with equality if and only if $q=1.$ Similar properties hold for hybrid Besov spaces.
\smallbreak
Throughout, we shall denote
\begin{equation}\label{eq:coeff}
\tilde\kappa:=\kappa/\nu,\quad\tilde\lambda:=\lambda/\nu,\quad\tilde\mu:=\mu/\nu\ \hbox{ with }\
\nu:=\lambda+2\mu.
\end{equation}

One can  state  our first main result : the global existence of solutions corresponding to small (critical) data
with \emph{estimates independent of $\eps$} in the case $\kappa>0.$
\begin{thm}\label{th:main1}
Assume that the initial data $(b_0^\eps, u_0^\eps, \theta_0^\eps)$ and that the potential term $V^\eps$ satisfy,
for a small enough constant $\eta$ \emph{depending only on
$\tilde\kappa$ and $\tilde\mu$}:
\begin{eqnarray}\label{eq:smalldata1}
&&\|b_0^\eps\|_{\tilde B^{\f32,-}_{\eps\nu}}+\|u_0^\eps\|_{\dot B^{\f12}_{2,1}}+
\|\theta_0^\eps\|_{\tilde{B}^{-\f12,+}_{\eps \nu}}
\leq \eta \nu,\\\label{eq:smalldata2}
&&\nu^{\frac12}\|\na V^\eps\|_{L^2(\tilde B^{\f32,-}_{\eps\nu})}+
\|V^\eps \|_{\tilde{L}^\infty(\tilde B^{\f32,-}_{\eps\nu})}+\|\pa_tV^\eps \|_{L^1(\tilde B^{\f32,-}_{\eps\nu})}
\leq\eta\nu.
\end{eqnarray}
Let $a_0^\eps:=b_0^\eps+V^\eps(0).$
Then System \eqref{eq:hNS} with
initial data $(1+\eps a_0^\eps, u_0^\eps,1+\eps\theta_0^\eps)$ has a unique global
solution $(a^\eps,u^\eps,\theta^\eps)$ (with $a^\eps=b^\eps+V^\eps$)
 which satisfies
 $$b^\eps\in\tilde C(\tilde B^{\f32,-}_{\eps\nu})\cap L^1(\tilde B^{\f32,+}_{\eps\nu}),\quad
 u^\eps\in\tilde C(\dot B^{\f12}_{2,1})\cap L^1(\dot B^{\f52}_{2,1}),\quad
\theta^\eps\in\tilde C(\tilde B^{-\f12,+}_{\eps\nu})\cap L^1(\tilde B^{\f32,+}_{\eps\nu})
$$
and, for a  constant $K$ depending only on $\tilde\kappa$  and $\tilde\mu,$
 $$\displaylines{
 \|b^\eps\|_{\tilde L^\infty(\tilde B^{\f32,-}_{\eps\nu})}+\nu\|b^\eps\|_{L^1(\tilde B^{\f32,+}_{\eps\nu})}
 +\|u^\eps\|_{\tilde L^\infty(\dot B^{\f12}_{2,1})}+\nu\|u^\eps\|_{L^1(\dot B^{\f52}_{2,1})}
 +\|\theta^\eps\|_{\tilde L^\infty(\tilde B^{-\f12,+}_{\eps\nu})}\hfill\cr\hfill+\nu\|\theta^\eps\|_{L^1(\tilde B^{\f32,+}_{\eps\nu})}\leq
K\bigl(\|b_0^\eps\|_{\tilde B^{\f32,-}_{\eps\nu}}+\|u_0^\eps\|_{\dot B^{\f12}_{2,1}}+
\|\theta_0^\eps\|_{\tilde{B}^{-\f12,+}_{\eps \nu}}+\|\pa_tV^\eps \|_{L^1(\tilde B^{\f32,-}_{\eps\nu})}\bigr).}
$$
 \end{thm}
 \begin{rmk}\label{r:smooth}
 Smoother data give rise to smoother solutions.
 For example if in addition to the above hypotheses, we have
\beno &&\eps\|b_0^\eps\|_{\tilde B^{\f52,-}_{\eps\nu}}
+\nu^{-1}\|(\theta_0^\eps,u_0^\eps)\|_{\tilde B^{\f32,-}_{\eps\nu}}
 +\eps\|\pa_tV^\eps \|_{L^1_t(\tilde B^{\f52,-}_{\eps\nu})}+\eps
 \|\nabla V^\eps\|_{L^2(\tilde B^{\f52,-}_{\eps\nu})}
\leq\eta,
\eeno then the above solution also satisfies
 \beno  && \eps\| b^\eps \|_{\tilde L^\infty(\tilde B^{\f52,-}_{\eps\nu})}
 +\nu^{-1}\|(\theta^\eps,u^\eps) \|_{\tilde L^\infty( \tilde B^{\f32,-}_{\eps\nu})}
 +\eps\nu\| b^\eps \|_{L^1( \tilde B^{\f52,+}_{\eps\nu})}
 +\eps\nu\|(u^\eps,\theta^\eps)\|_{L^1(\tilde B^{\f72,-}_{\eps\nu})}\leq K\eta.\eeno
  \end{rmk}
 Next, combining this result with Strichartz estimates, we shall prove the following
 result of convergence to the Boussinesq system.
 \begin{thm}\label{th:main2}
Consider a family of data $(b_0^\eps,u_0^\eps,\theta_0^\eps,V^\eps)_{\eps>0}$ satisfying the conditions of
Theorem \ref{th:main1} with in addition
\begin{equation}\label{eq:unif}
\begin{aligned}
M_0:=\sup_{\eps>0}\bigl(
\|b_0^\eps\|_{\tilde B^{\f32,-}_{\eps\nu}}&+\|u_0^\eps\|_{\dot B^{\f12}_{2,1}}+
\|\theta_0^\eps\|_{\tilde{B}^{-\f12,+}_{\eps \nu}}\\&+
\nu^{\f12}\|\nabla V^\eps\|_{L^2(\tilde B^{\f32,-}_{\eps\nu})}+
\|V^\eps \|_{\tilde L^\infty(\tilde B^{\f32,-}_{\eps\nu})}+\|\pa_tV^\eps \|_{L^1(\tilde B^{\f32,-}_{\eps\nu})}\bigr)
\leq\eta\nu.\end{aligned}
\end{equation}
Let  $q^\eps:=(\theta^\eps+b^\eps)/\sqrt2$ and $\Theta^\eps:=(\theta^\eps-b^\eps)/\sqrt2.$
Assume that $(\cP u^\eps_0,\Theta^\eps_0,V^\eps)$ converges (in the sense of distributions)
to  some triplet   $(v_0,\Theta_0,V)$ such that
$$
v_0\in\dot B^{\f12}_{2,1},\quad \Theta_0\in\dot B^{\f12}_{2,1},\quad
\nabla V\in L^2(\dot B^{\f12}_{2,1}),\quad\pa_tV\in L^1(\dot B^{\f12}_{2,1}).$$
Then the following properties hold true~:
\begin{enumerate}
\item  System \eqref{eq:hNS} with
initial data $(1+\eps a_0^\eps, u_0^\eps,1+\eps\theta_0^\eps)$ has a unique global  solution
with the properties described in Theorem \ref{th:main1};
\item  Boussinesq system \eqref{eq:BouS} admits a unique global solution
$(v,\Theta)$ in $\tilde \cC(\dot B^{\f12}_{2,1})\cap L^1(\dot B^{\f52}_{2,1})$ satisfying
for some constant $K=K(\tilde\kappa,\tilde\mu)$:
$$
\|(v,\Theta)\|_{\tilde L^\infty(\dot B^{\f12}_{2,1})}+\nu  \|(v,\Theta)\|_{L^1(\dot B^{\f52}_{2,1})}
\leq K\bigl(\|(v_0,\Theta_0)\|_{\dot B^{\f12}_{2,1}}+\|\pa_tV\|_{L^1(\dot B^{\f12}_{2,1})}\bigr).
$$
\item The functions $q^\eps$ and $\cQ u^\eps$ go to $0$ in the following meaning
for all $p\in[2,\infty]$ and $s\in[-1/2+4/p,3/p]$:
$$
\nu^{\f12}\|q^\eps\|_{\tilde{L}^2(\tilde B^{s-1,+}_{p,\eps\nu})}+
\nu^{\f12}\|\cQ u^\eps\|_{\tilde{L}^2(\dot B^{s}_{p,1})} \leq K(\eps\nu)^{\f3p-s}M_0.
$$
\item The couple $(\cP u^\eps,\Theta^\eps)$ tends to $(v,\Theta)$ in the following meaning
for all $p\in[2,\infty]$ and $s\in[-1/2+4/p,3/p]$ with $s>1/2:$
$$\displaylines{\nu^{1/2}\|\dT\|_{\tilde L^2(\tilde B^{s-1,+}_{p,\eps\nu})}
+\|\dT\|_{\tilde L^\infty(\tilde B^{s-2,+}_{p,\eps\nu})}+
\nu\|\dv\|_{L^1(\tilde B^{s,+}_{p,\eps\nu})}+ \|\dv\|_{\tilde L^\infty(\tilde B^{s-2,+}_{p,\eps\nu})}\hfill\cr\hfill
\leq C\bigl(\|(\dT_0,\dv_0)\|_{\tilde B^{s-2,+}_{p,\eps\nu}}+\|\pa_t\dV\|_{L^1(\tilde B^{s-2,+}_{p,\eps\nu})+L^2(\tilde B^{s-3,+}_{p,\eps\nu})}
+ M_0^2\eps^{\f3p-s}+M_0\|\nabla\dV\|_{L^2(\dot B^{s-1}_{p,1})}
\bigr)}
$$
with $\dT:=\Theta^\eps-\Theta,$ $\dv:=\cP u^\eps-v,$  $\dV:=V^\eps-V$
and $C=C(\tilde\mu,\tilde\kappa,s,p).$
 \end{enumerate}
 \end{thm}
 \begin{rmk}\label{r:smoothbis} If the data are smoother, e.g. as in Remark \ref{r:smooth} then
 the results of convergence hold for stronger norms. For instance, it may be shown that
$(\cQ u^\eps,q^\eps)\to 0 \ \hbox{ in }\ \tilde L^2(\dot B^{\f4p-\f12}_{p,1}),$
that $\cP u^\eps\to v$ in $L^1(\dot B^{\f4p+\f12}_{p,1})\cap  \tilde L^\infty(\dot B^{\f4p-\f32}_{p,1}),$
and that  $\Theta^\eps\to \Theta$ in $ \tilde L^2(\dot B^{\f4p-\f12}_{p,1})\cap  \tilde L^\infty(\dot B^{\f4p-\f32}_{p,1}),$
 with the  decay rate $\eps^{\f12-\f1p}.$
\end{rmk}
Let us finally state our main global existence and convergence result
for nonconducting fluids.
\begin{thm}\label{th:main3}
Assume that the initial data $(a_0^\eps, u_0^\eps, \mathcal{R}_0^\eps)$ and the force term $V^\eps$ verify that
$$
\begin{array}{l}
C_0^\eps:=\|(a_0^\eps,u_0^\eps,\cR_0^\eps)\|_{\dot B^{\f12}_{2,1}}
\!+\!(\eps\nu)^3\|(a_0^\eps,\cR^\eps_0)\|_{\dot B^{\f72}_{2,1}}
+(\eps\nu)^2\|u_0^\eps\|_{\dot B^\f52_{2,1}}\\\hspace{7.8cm}
+ \|\pa_tV^\eps \|_{L^1(\dot B^\f12_{2,1})}\!+\!(\eps\nu)^3\|\pa_tV^\eps \|_{L^1(\dot B^\f72_{2,1})}\leq\eta\nu,\\[2ex]
\|V^\eps \|_{\tilde L^\infty(\dot B^\f12_{2,1})}\!+(\eps\nu)^3\|V^\eps \|_{\tilde L^\infty(\dot B^\f72_{2,1})}
  \!+\nu\|V^\eps \|_{L^1(\dot  B^\f52_{2,1})}  +\nu(\eps\nu)^2\|V^\eps \|_{L^1(\dot B^\f92_{2,1})}
\leq \eta\nu\end{array}
$$ where the constant $\eta$ is sufficiently small and depends
only on  $\tilde\mu.$
\smallbreak
Then System \eqref{eq:nhNS} admits a unique global  solution $(a^\eps,u^\eps,\mathcal{R}^\eps)$
which satisfies
$$
a^\eps\!\in\!\tilde C(\dot B^{\f12}_{2,1}\cap\dot B^{\f72}_{2,1}),\quad\!\!
u^\eps\!\in\!\tilde C(\dot B^{\f12}_{2,1}\cap\dot B^{\f52}_{2,1})\cap L^1(\dot B^{\f52}_{2,1}\cap\dot B^{\f92}_{2,1}),\quad\!\!
\cR^\eps\!\in\!\tilde C(\dot B^{\f12}_{2,1}\cap\dot B^{\f72}_{2,1})\cap L^1(\tilde B^{\f32,+}_{\eps\nu}\cap \tilde B^{\f72,+}_{\eps\nu})
$$
and, for some constant $K$ depending only on $\tilde\mu,$
\beno &&\|(a^\eps,\cR^\eps)\|_{\tilde L^\infty(\dot B^{\f12}_{2,1})}
+(\eps\nu)^3\|(a^\eps,\cR^\eps)\|_{\tilde L^\infty(\dot B^{\f72}_{2,1})}
+\|u^\eps\|_{\tilde L^\infty(\dot B^{\f12}_{2,1})}+(\eps\nu)^2\|u^\eps \|_{\tilde L^\infty(\dot B^{\f52}_{2,1})}
\\&&\quad
+\nu\|\cR^\eps\|_{L^1(\tilde B^{\f32,+}_{\eps\nu})}
+\nu(\eps\nu)^2\|\cR^\eps\|_{L^1(\tilde B^{\f72,+}_{\eps\nu})}
+\nu\|u^\eps\|_{L^1(\dot B^\f52_{2,1})}+\nu(\eps\nu)^2\|u^\eps \|_{L^1(\dot B^\f92_{2,1})}
\leq KC_0^\eps.
 \eeno

Suppose in addition that $\Theta^\eps_0\rightarrow \Theta_0,$ that $\cP u_0^\eps\rightarrow v_0$
and that $V^\eps\rightarrow V$ with
\begin{equation}\label{eq:smallnBouS}
\|v_0\|_{\dot B^{\f12}_{2,1}}+\|\nabla V\|_{L^1(\dot B^{\f32}_{2,1})}+\|\Theta_0\|_{\dot B^{\f12}_{2,1}}\leq \eta\mu.
\end{equation}
Then  the corresponding limit Boussinesq system  \eqref{eq:nBouS}
admits a unique global solution $(\Theta,v)$ in $\tilde C(\dot B^{\f12}_{2,1})\times\bigl( \tilde C(\dot B^{\f12}_{2,1})\cap L^1(\dot B^{\f52}_{2,1})\bigr).$
Furthermore we have
\begin{equation}\label{eq:nBouS1}
\|(\Theta,v)\|_{\tilde L^\infty(\dot B^{\f12}_{2,1})}+\mu\|v\|_{L^1(\dot B^{\f52}_{2,1})}\leq
K\|(\Theta_0,v_0)\|_{\dot B^{\f12}_{2,1}}.
\end{equation}
In addition, if $C_0^\eps$ is bounded by some constant $C_0$ when $\eps$ goes to $0$
then $(\cQ u^\eps,\cR^\eps)$ goes to zero with the following rates of convergence for all $p\in[2,\infty):$
\begin{align}\label{eq:stric0a}
\|(\cQ u^\eps,\cR^\eps)\|_{\tilde L^{\frac{2p}{p-2}}(\dot B^{\frac2p-\frac12}_{p,1})}&\leq KC_0\eps^{\frac12-\frac1p}\\
\nu^{\f12}\|(\cQ u^\eps,\cR^\eps)\|_{\tilde L^{2}(\dot B^{s}_{p,1})}&\leq
KC_0(\eps\nu)^{\f3p-s}\quad\hbox{ if }\
s\in[-1/2+4/p,3/p].
\end{align}
Finally, if $\Theta_0^\eps$ and $\cP u_0^\eps$ are independent\footnote{The reader may refer to Inequalities \eqref{eq:final1}
and \eqref{eq:final2} for the general case.} of $\eps$ then for all $p$ and $s$ as above (with in addition $s>1/2$), and $T>0,$
$$
\Theta^\eps-\Theta\rightarrow 0\ \hbox{ in }\ \tilde C_T(\dot B^{s-2}_{p,1})\quad\hbox{and}\quad
\cP u^\eps-v\rightarrow0\ \hbox{ in }\
\tilde C_T\bigl(\dot B^{s-1}_{p,1}+\dot B^{s-2}_{p,1}\bigr)\cap \bigl(\tilde L^2_T(\dot B^s_{p,1})+L^1_T(\dot B^s_{p,1})\bigr),
$$
and the rate of convergence is  $\eps^{\f3p-s}.$
\end{thm}

The above statements deserve some comments:
\begin{enumerate}
\item  In this paper, for simplicity, we focussed on the physical dimension $3.$ However
similar statements may be established in any dimension $d\geq2.$
\item In the case of large data, we expect, as for the isentropic Navier-Stokes equations studied in \cite{D3},
the lifespan of the solutions to \eqref{eq:hNS} to tend to that of the limit Oberbeck-Boussinesq equations.
Global existence for the limit equations should entail  global existence
for \eqref{eq:hNS} with small $\eps,$ if $\kappa>0.$ This is of particular interest in dimension two, as the limit equations
are globally well-posed for any data with the above smoothness.
We reserve this study to future works.
 \item  We also reserve the case of other boundary conditions, in particular the periodic ones, to future
works. We want to point out that the global existence statements (that is Theorem \ref{th:main1}
 as well as the first part of Theorem
\ref{th:main3}) extend to that case. At the same time, no dispersive inequalities
are available, hence  the approach
for proving convergence is expected to be completely different,
provided based  on the filtering method, as in the isentropic case \cite{D4}.
\end{enumerate}

We end this section by explaining the general strategy for the proof of convergence.
The first step consists in proving   uniform global a priori estimates. This in fact corresponds to
the statement of Theorem \ref{th:main1} and to the first part of Theorem \ref{th:main3}.
We shall see that the proof reduces  to the case $\eps=1$ after suitable rescaling of the equations.
Then, proving convergence requires two steps : first we
establish that the oscillating part of the solution converges to $0$ (this relies on Strichartz estimates),
and next  establish strong convergence to Oberbeck-Boussinesq for the incompressible modes.
Note that, owing to the fact that only small solutions are considered, we do not need to resort to
bootstrap arguments.


\section{Global existence and convergence in the case $\kappa>0$}\label{s:kappa}

Let us first notice that performing the change of unknown\footnote{Recall that $\nu=\lambda+2\mu$}:
\begin{equation}\label{eq:change}
(b,u,\theta)(t,x):=\eps (b^\eps,u^\eps,\theta^\eps)(\eps^2\nu t,\eps\nu x)
\end{equation}
and the change of data
\begin{equation}\label{eq:changedata}
(b_0,u_0,\theta_0)(x):= \eps(b^\eps_0,u^\eps_0,\theta^\eps_0)(\eps\nu x)
\quad\hbox{and}\quad
\widetilde{V}(t,x):=\eps V^\eps(\eps^2\nu t,\eps\nu x)
 \end{equation}
reduces the study to the case $\nu=1$ and $\eps=1.$
Indeed it is obvious that $(b^\eps,u^\eps,\theta^\eps)$ satisfies \eqref{eq:hNS}
  if and only if $(b,u,\theta)$ satisfies the same system with $\eps=1,$
  Lam\'e  coefficients
 $(\tilde\lambda,\tilde\mu):=\nu^{-1}(\lambda,\mu)$ and
 heat conductivity $\tilde\kappa:=\nu^{-1}\kappa,$
 provided the data have been changed according to \eqref{eq:changedata}.
 This change of variables has the desired effect
 on the norms that are used in Theorem \ref{th:main1}. For example,
  we have, up to a constant independent of $\eps$ and $\nu,$
 $$\begin{array}{l}
  \|b_0\|_{\tilde B^{\f32,-}_1}=\nu^{-1}\|b^\eps_0\|_{\tilde B^{\f32,-}_{\eps\nu}},
  \quad   \|u_0\|_{\dot B^{\f12}_{2,1}}=\nu^{-1}\|u_0^\eps\|_{\dot B^{\f12}_{2,1}},
  \quad  \|\theta_0\|_{\tilde B^{-\f12,+}_1}=\nu^{-1}\|\theta^\eps_0\|_{\tilde B^{-\f12,+}_{\eps\nu}},\\[2ex]
 \|\nabla \widetilde{V}(t,\cdot)\|_{\tilde B^{\f32,-}_1}=\eps\|\nabla V^\eps(\eps^2\nu t,\cdot)\|_{\tilde B^{\f32,-}_{\eps\nu}}
 \quad\hbox{and}\quad\|\pa_t\widetilde{V}(t,\cdot)\|_{\tilde B^{\f32,-}_{1}}
= \eps^2\|\pa_tV^\eps(\eps^2\nu t,\cdot)\|_{\tilde B^{\f32,-}_{\eps\nu}},\end{array}
 $$
 hence
 $$
  \|\nabla \widetilde{V}\|_{L^2(\tilde B^{\f32,-}_1)}=\nu^{-\f12}\|\nabla V^\eps\|_{L^2(\tilde B^{\f32,-}_{\eps\nu})}
  \quad\hbox{and}\quad
   \|\pa_t \widetilde{V}\|_{L^1(\tilde B^{\f32,-}_1)}=\nu^{-1}\|\pa_t V^\eps\|_{L^1(\tilde B^{\f32,-}_{\eps\nu})} .
 $$
 Consequently, in order to prove  Theorem \ref{th:main1}, it is suffices to consider the case $\nu=1$ and
  $\eps=1.$  We shall resume to the original variables only at the end of this section, for getting the convergence results
  of Theorem \ref{th:main2}.


\subsection{The linearized system}

In the case $\eps=\nu=1,$ the linearized equations about $(0,0,0)$ read
 \begin{equation}
  \left\{
    \begin{aligned}
      & \pa_t b+\div u=0, \\
      &\pa_t u - \tilde\mu\Delta u-(\tilde\lambda+\tilde\mu)\nabla\div u+\nabla(b+\theta)  =0,\\
      &\pa_t \theta  +\div u-\tilde\kappa \Delta \theta=0.
         \end{aligned}
  \right.
  \label{eq:linearized}
  \end{equation}
  We aim at proving energy type estimates for $(b,u,\theta).$
Roughly speaking, we shall exhibit a low frequency parabolic type smoothing for all the components
of the solution whereas, in high frequency,  only $(u,\theta)$  will  experience a parabolic smoothing.
As for  $b,$ it will be  damped with no gain of regularity whatsoever.
 Throughout our proof (which will require several steps) we shall
also pinpoint  where  one has to work in different level of regularities to get the aforementioned
features of the system.

Let us first notice that  the gradient terms in the velocity equation involve only the potential part of the velocity.
    More precisely, setting
 $d:=\Lambda^{-1}\div u$ (with  $\Lambda^s:=|D|^s$)  and $w:=\cP u=u+\na(-\Delta^{-1})\div u,$ we get
 \begin{equation}
  \left\{
    \begin{aligned}
      & \pa_t b+\Lambda d=0, \\
      &\pa_t d - \Delta d-\Lambda(b+\theta)  =0,\\
      &\pa_t \theta  +\Lambda d -\tilde\kappa \Delta \theta=0,\\
      &\pa_tw-\tilde\mu\Delta w=0.
         \end{aligned}
  \right.
  \label{eq:linearizedbis}
  \end{equation}
 As the last equation is the standard heat equation with constant diffusion, we focus on the proof  of
 estimates for the first three equations.
 After localization by means of the homogeneous Littlewood-Paley decomposition $(\ddj)_{j\in\Z},$ the obtained system reads
   \begin{equation}
  \left\{
    \begin{aligned}
      & \pa_t b_j+\Lambda d_j=0, \\
      &\pa_t d_j - \Delta d_j-\Lambda(b_j+\theta_j)  =0,\\
      &\pa_t \theta_j +\Lambda d_j -\tilde\kappa \Delta \theta_j=0
         \end{aligned}
  \right.
  \label{eq:bdtheta}
  \end{equation}
with $b_j:=\ddj b,$ $d_j:=\ddj d$ and $\theta_j:=\ddj\theta.$

 \subsubsection*{Step 1: Basic Energy Estimate for $(b, d, \theta)$}
 Owing to the antisymmetric structure of the first order terms in \eqref{eq:bdtheta}, we readily get
 \begin{equation}\label{be1}
\f12\f{d}{dt}\big[\|b_j\|_{L^2}^2+\|d_j\|_{L^2}^2+\|\theta_j\|_{L^2}^2\big]+\|\Lambda d_j\|_{L^2}^2+\tilde\kappa\|
\Lambda \theta_j\|_{L^2}^2=0.
\end{equation}

 \subsubsection*{Step 2: Improved Energy Estimate for $(b, d, \theta)$}

 We want to track the decay  properties of $b.$ For that we notice that the auxiliary function $\Lambda b-d$ satisfies:
 $$\pa_t[\Lambda b_j-d_j] +\Lambda(b_j+\theta_j)=0.$$
 Hence  taking the $L^2$ inner product with $\Lambda b_j-d_j$ yields
  $$\displaylines{\f12\f{d}{dt}\|\Lambda b_j-d_j \|_{L^2}^2
  +\|\Lambda b_j\|_{L^2}^2+\bigl(\Lambda \theta_j|\Lambda b_j\bigr)_{L^2}-\bigl((b_j+\theta_j)| \Lambda d_j\bigr)_{L^2}=0,}
  $$
from which we deduce that
 \begin{equation}\label{be3}   \f12\f{d}{dt}\bigg[\|\Lambda b_j-d_j \|_{L^2}^2+\|b_j\|_{L^2}^2 +\|\theta_j\|_{L^2}^2
  \bigg]+\|\Lambda b_j\|_{L^2}^2+\big(\Lambda \theta_j|\Lambda b_j \big)_{L^2}
  +\tilde\kappa \|\Lambda \theta_j  \|_{L^2}^2=0. \end{equation}

 Putting  \eqref{be1} and \eqref{be3} together, we thus get for any $\alpha\geq0,$
  \begin{multline}\label{be2}
  \f12\f{d}{dt}\bigg[\alpha\|d_j\|_{L^2}^2+\|\Lambda b_j-d_j] \|_{L^2}^2+(1+\alpha)\|b_j\|_{L^2}^2
  +(1+\alpha)\|\theta_j\|_{L^2}^2\biggr]\\+\|\Lambda b_j\|_{L^2}^2+\big(\Lambda \theta_j|\Lambda b_j \big)_{L^2}
  +\tilde\kappa (1+\alpha) \|\Lambda \theta_j  \|_{L^2}^2+\alpha\|\Lambda d_j\|_{L^2}^2=0.
  \end{multline}
  Let us denote
\begin{eqnarray}\label{eq:fj}
&& f_j^2:= \alpha \|d_j\|_{L^2}^2+(1+\alpha)\|b_j\|_{L^2}^2
 +\|\Lambda b_j-d_j \|_{L^2}^2  +(1+\alpha)\|\theta_j\|_{L^2}^2,\\\label{eq:Hj}
 &&H_j^2:=\frac12\|\Lambda b_j\|_{L^2}^2+\alpha\|\Lambda d_j\|_{L^2}^2
  +\biggl(\tilde\kappa(1+\alpha)-\frac12\biggr) \|\Lambda \theta_j  \|_{L^2}^2.
  \end{eqnarray}
 Then combining   \eqref{be2} with the following Young inequality:
 $$
 \big|\big(\Lambda \theta_j|\Lambda b_j \big)_{L^2}\bigr|\leq\frac{1}{2}\|\Lambda\theta_j\|_{L^2}^2+\frac12\|\Lambda b_j\|_{L^2}^2,
 $$
 implies that
\begin{equation}\label{be4}
\f12\f{d}{dt}f_j^2+ H_j^2\leq 0.
 \end{equation}
 Let us notice that
 $$
 f_j^2=(\alpha+1)\|(b_j,d_j,\theta_j)\|_{L^2}^2+\|\Lambda b_j\|_{L^2}^2-2(\Lambda b_j|d_j)_{L^2}.
 $$
 Therefore, because
 $$
 2|(\Lambda b_j|d_j)_{L^2}|\leq\frac23\|\Lambda b_j\|_{L^2}^2+\frac32\|d_j\|_{L^2}^2,
 $$
  we have
 \begin{equation}\label{eq:equivfj}
 \Bigl(\alpha-\frac12\Bigr)\|d_j\|_{L^2}^2+\frac13\|\Lambda b_j\|_{L^2}^2
 \leq f_j^2-(\alpha+1)\|(b_j,\theta_j)\|_{L^2}^2\leq  \Bigl(\alpha+\frac52\Bigr)\|d_j\|_{L^2}^2+\frac53\|\Lambda b_j\|_{L^2}^2.
 \end{equation}
 Let us first assume that $\tilde\kappa\leq1.$ Then
we take $\alpha=2/\tilde\kappa-1$ and \eqref{eq:equivfj} thus implies that
$$
f_j^2\approx\left\{\begin{array}{lll} \tilde\kappa^{-1}\|(b_j,d_j,\theta_j)\|_{L^2}^2&\hbox{ if }&\tilde\kappa2^{2j}\leq1,\\
 \tilde\kappa^{-1}\|(d_j,\theta_j)\|_{L^2}^2+\|\Lambda b_j\|_{L^2}^2&\hbox{ if }&\tilde\kappa2^{2j}\geq1.\end{array}\right.
 $$
At the same time, we have
$$
H_j^2\gtrsim \left\{\begin{array}{lll} 2^{2j}\|(b_j,d_j,\theta_j)\|_{L^2}^2&\hbox{ if }&\tilde\kappa2^{2j}\leq1,\\
 \tilde\kappa^{-1}\|(d_j,\theta_j)\|_{L^2}^2+\|\Lambda b_j\|_{L^2}^2&\hbox{ if }&\tilde\kappa2^{2j}\geq1.\end{array}\right.
 $$
Therefore, one may easily conclude that for some (universal) constant $c\in(0,1],$
\begin{eqnarray}\label{be5l}
\|(b_j,d_j,\theta_j)(t)\|_{L^2}\lesssim e^{-c\tilde\kappa2^{2j}t}\|(b_j,d_j,\theta_j)(0)\|_{L^2}\quad\hbox{if }\ 2^{2j}\tilde\kappa\leq 1,\\
\label{be5h}
\|(\tilde\kappa\Lambda b_j,d_j,\theta_j)(t)\|_{L^2}\lesssim e^{-ct}\|(\tilde\kappa\Lambda b_j,d_j,\theta_j)(0)\|_{L^2}
\quad\hbox{if }\ 2^{2j}\tilde\kappa\geq 1.
\end{eqnarray}
Let us now assume that $\tilde\kappa\geq1.$ Then we take $\alpha=1$ so that  following the above computations after replacing
everywhere $\tilde\kappa$ by $1,$ it is easy to conclude that
\begin{eqnarray}\label{be6l}
\|(b_j,d_j,\theta_j)(t)\|_{L^2}\lesssim e^{-c2^{2j}t}\|(b_j,d_j,\theta_j)(0)\|_{L^2}\quad\hbox{if }\  j\leq 0,\\
\label{be6h}
\|(\Lambda b_j,d_j,\theta_j)(t)\|_{L^2}\lesssim e^{-ct}\|(\Lambda b_j,d_j,\theta_j)(0)\|_{L^2}
\quad\hbox{if }\  j\geq0.
\end{eqnarray}
Therefore, denoting $\check\kappa =\min(1,\tilde\kappa)$ and
putting together \eqref{be5l}, \eqref{be5h}, \eqref{be6l} and \eqref{be6h}, we end up with
\begin{equation}\label{be6}\begin{array}{lll}
\|(b_j,d_j,\theta_j)(t)\|_{L^2}\lesssim e^{-c\check\kappa2^{2j}t}\|(b_j,d_j,\theta_j)(0)\|_{L^2} &\!\!\hbox{if}\!\!& 2^{2j}\check\kappa\leq 1,\\[2ex]
\|\Lambda b_j(t)\|_{L^2}+\check\kappa^{-1}\|(d_j,\theta_j)(t)\|_{L^2}
\lesssim e^{-ct}\bigl(\|\Lambda b_j(0)\|_{L^2}\!+\!\|\check\kappa^{-1}(d_j,\theta_j)(0)\|_{L^2}\bigr)
&\!\!\hbox{if}\!\!& 2^{2j}\check\kappa\geq 1.
\end{array}\!\!\!\!\!\!\!\!
\end{equation}
For reasons that will appear more clearly in the following steps, it is suitable to work \emph{with one less derivative}
in the high frequency regime. Now from the second inequality of \eqref{be6} and
Bernstein inequality, we get for $2^{2j}\check\kappa\geq1,$
\begin{equation}\label{be6g}
\|\check\kappa b_j(t)\|_{L^2}\!+\!\|\Lambda^{-1}(d_j,\theta_j)(t)\|_{L^2}
\lesssim e^{-ct}\bigl(\|\check\kappa b_j(0)\|_{L^2}\!+\!\|\Lambda^{-1}(d_j,\theta_j)(0)\|_{L^2}\bigr).
\end{equation}

\subsubsection*{Step 3: Parabolic smoothing for $\theta$. }
We here aim at tracking the high-frequency parabolic smoothing for $\theta.$  For that, we
rewrite the last two equations of \eqref{eq:bdtheta} as follows
$$
\left\{\begin{array}{l}
\pa_t\Lambda^{-1}d_j-\Delta(\Lambda^{-1}d_j)-\theta_j=b_j,\\[1ex]
\pa_t\Lambda^{-1}\theta_j-\tilde\kappa\Delta(\Lambda^{-1}\theta_j)+d_j=0.
\end{array}\right.
$$
Then applying a direct energy method,  we readily get
$$
 \f12\f{d}{dt}\bigg(\|\Lambda^{-1} d_j\|_{L^2}^2 + \|\Lambda^{-1}\theta_j\|_{L^2}^2 \bigg)+\|d_j\|_{L^2}^2 +\tilde\kappa\|\theta_j\|_{L^2}^2
 =(b_j|\Lambda^{-1} d_j).
 $$
 Therefore, performing a time integration yields
 $$
 \|\Lambda^{-1}(d_j,\theta_j)(t)\|_{L^2}+c\check\kappa\int_0^t \|\Lambda(d_j,\theta_j)\|_{L^2}\,d\tau
 \leq  \|\Lambda^{-1}(d_j,\theta_j)(0)\|_{L^2}+\int_0^t\|b_j\|_{L^2}\,d\tau,
 $$
 and taking advantage of the second inequality of \eqref{be6} eventually leads to
 \begin{equation}
  \label{be7}
   \|\Lambda^{-1}(d_j,\theta_j)(t)\|_{L^2}+\check\kappa\int_0^t \|\Lambda(d_j,\theta_j)\|_{L^2}\,d\tau
 \lesssim      \|b_j(0)\|_{L^2}+ \check\kappa^{-1}\|\Lambda^{-1}(d_j,\theta_j)(0)\|_{L^2}
   \end{equation}
in the high frequency regime, that is whenever $2^{j}\sqrt{\check\kappa}\geq1.$

 \subsubsection*{Step 4: Parabolic smoothing  for $d$. }

Given that
$$
\pa_td_j-\Delta d_j=\Lambda(b_j+\theta_j),
$$
one may write that
$$
\|d_j(t)\|_{L^2}+c2^{2j}\int_0^t\|d_j\|_{L^2}\,d\tau\leq \|d_j(0)\|_{L^2}+\int_0^t\|\Lambda(b_j,\theta_j)\|_{L^2}\,d\tau.
$$
The previous steps ensure that, for $2^j\sqrt{\check\kappa}\geq1,$
$$
\begin{array}{lll}
\displaystyle\int_0^t\|\Lambda b_j\|_{L^2}\,d\tau&\lesssim& \|\Lambda b_j(0)\|_{L^2}+\check\kappa^{-1}\|(d_j,\theta_j)(0)\|_{L^2},\\[1ex]
\Int_0^t\|\Lambda \theta_j\|_{L^2}\,d\tau&\lesssim& \check\kappa^{-1}\|b_j(0)\|_{L^2}+\check\kappa^{-2}\|\Lambda^{-1}(d_j,\theta_j)(0)\|_{L^2}.
\end{array}
$$
Therefore we have
    \begin{multline}
    2^{2j}\!\int_0^t\!\|d_j\|_{L^2}\,d\tau
    \lesssim    \|\Lambda b_j(0)\|_{L^2} +\check\kappa^{-1}\|b_j(0)\|_{L^2}\\+ \check\kappa^{-2} \|\Lambda^{-1}(d_j,\theta_j)(0)\|_{L^2}
  +\check\kappa^{-1}\|(d_j,\theta_j)(0)\|_{L^2}.\!
   \label{be8}\end{multline}

\subsubsection*{Step 5: Final a priori estimate for $(b,u,\theta)$}

Putting together inequalities \eqref{be6}, \eqref{be7} and \eqref{be8} and using  the standard properties of the
heat equation (as regards $w$), we get if $j\leq0$:
\begin{equation}\label{eq:estlow}
\|(b_j,u_j,\theta_j)(t)\|_{L^2}+2^{2j}\Int_0^t\|(b_j,u_j,\theta_j)\|_{L^2}\,d\tau
\leq C\|(b_j,u_j,\theta_j)(0)\|_{L^2},
\end{equation}
and, if $j\geq0$:
\begin{equation}\label{eq:esthigh}
\|(2^jb_j,u_j,2^{-j}\theta_j)(t)\|_{L^2}+\Int_0^t\|(2^{j}b_j,2^{2j}u_j,2^j\theta_j)\|_{L^2}\,d\tau
\leq C\|(2^jb_j,u_j,2^{-j}\theta_j)(0)\|_{L^2}.
\end{equation}
The above constant $C$ depends only on $\check\mu$ and $\check\kappa.$


\subsection{A priori estimates for the paralinearized system}

As pointed out in the previous subsection (see in particular \eqref{eq:esthigh}),
there is no gain of regularity for $b$ throughout the evolution (only damping in fact).
Therefore, the convection term $v\cdot\nabla b$ cannot
just be considered as a source term, tractable by Duhamel formula,
for the presence of $\nabla b$ will induce a loss of one derivative in the estimates.

At the same time, at the level of $L^2$ estimates, this convection term is
rather harmless provided $\div v$ is in $L^1(\R^+;L^\infty)$
(it is only a matter of integrating by parts).
The natural idea is thus to keep the convection terms in the linearized equations\footnote{We
keep \emph{all} the terms just for questions of symmetry, but only $v\cdot\nabla b$ may cause a loss of derivative.}
and to resume to  the method of the previous
paragraph. As however the Littlewood-Paley localization operator $\ddj$
\emph{does not} commute with the material derivative $(\pa_t+v\cdot\nabla),$
it is convenient to keep only the `bad' part of the convection term, that
is the one which does induce a loss of one derivative.
In order to better explain what we mean, we have to give a short
presentation of Bony's decomposition (first introduced in \cite{Chon})
and paraproduct calculus.
The paraproduct is the bilinear operator
defined on the set of couples  of tempered distributions, by
$$
T_f g :=\sum_{j}\dot S_{j-1}f\, \ddj g\quad\hbox{with} \quad \dot S_{j-1}:=\chi(2^{-(j-1)}D).
$$
The (formal) Bony decomposition of the product $fg$ reads
$$
fg=T_fg+T'_gf.
$$
The basic idea is that the term $T_fg$ is always defined but cannot be more regular than $g,$
and that under suitable assumptions  the other term $T'_gf$ is more regular.
If we look at the convection term, the `bad' part that may
cause a loss of one derivative and has to be included in the linear analysis
is thus (with the summation convention over repeated indices)  $T_{u^k}\pa_k b.$
This motivates us to extend the analysis of the previous subsection
to the following `paralinearized' system:
 \begin{equation}
  \left\{
    \begin{aligned}
      & \pa_t b+\Lambda d+T_{v^k}\pa_k b=B, \\
      &\pa_t d + T_{v^k}\pa_k d- \Delta d-\Lambda(b+\theta)  =D,\\
      &\pa_t \theta  +\Lambda d +T_{v^k}\pa_k \theta-\tilde\kappa \Delta \theta=G,\\
      &\pa_t w + T_{v^k}\pa_k w-\tilde\mu\Delta w=W,
    \end{aligned}
  \right.
  \label{eq:lhNS}
  \end{equation}
where the source terms $B,$ $D,$ $G,$ $W$ and the vector field $v$ are given.
\begin{prop}\label{p:paralinear}
Let  $\cV(t):=\int_0^t \|\na v\|_{L^\infty}\,d\tau$. For all $s\in\R,$ there exists a constant $K$ depending only on
$\tilde\mu,$ $\check\kappa,$
and a universal constant $C$  such that  the following inequality holds true:
$$\displaylines{
\|b\|_{\tilde L^\infty_t(\tilde B^{s+1,-}_1)}+\|(d,w)\|_{\tilde L^\infty_t(\dot B^s_{2,1})}+\|\theta\|_{\tilde L^\infty_t(\tilde B^{s-1,+}_1)}
+\int_0^t\bigl(\|b\|_{\tilde B^{s+1,+}_1}+\|(d,w)\|_{\dot B^{s+2}_{2,1}}+\|\theta\|_{\tilde B^{s+1,+}_1}\bigr)\,d\tau
 \hfill\cr\hfill\leq Ke^{C\cV(t)}\biggl(\|b_0\|_{\tilde B^{s+1,-}_1}+\|(d_0,w_0)\|_{\dot B^s_{2,1}}+\|\theta_0\|_{\tilde B^{s-1,+}_1}
\hfill\cr\hfill +\int_0^te^{-C\cV(\tau)}\bigl(\|B\|_{\tilde B^{s+1,-}_1}+\|(D,W)\|_{\dot B^s_{2,1}}+\|G\|_{\tilde B^{s-1,+}_1}\bigr)\,d\tau\biggr)\cdotp}
 $$
\end{prop}
\begin{proof}
Compared to the study of the previous subsection,
the main additional difficulty lies in the
paraconvection terms. Indeed, the source terms may be easily dealt with by means of the Duhamel formula.

The paraconvection terms may be handled thanks to the following  inequality:
\begin{equation}\label{eq:paraconv}
\bigl|\bigl(\phi(2^{-j}D)(T_{v^k}\pa_k z)|\phi(2^{-j}D)z\bigr)_{L^2}\bigr|\leq C\|\nabla v\|_{L^\infty}
\|\phi(2^{-j}D)z\|_{L^2}\sum_{|j'-j|\leq N}\|\phi(2^{-j'}D)z\|_{L^2}
\end{equation}
which holds true for any smooth function $\phi$ with compact support away from the origin and large enough integer $N$
depending only on $\Supp\phi$ and $\Supp\varphi.$
\smallbreak
Let us justify \eqref{eq:paraconv}.
We fix some integer $N$ so that
$$\Supp \phi(2^{-j}\cdot)\cap \Supp \bigl(\chi(2^{-j'}\cdot)*\varphi(2^{-j'}\cdot)\bigr)=\emptyset
\quad\hbox{whenever }\  |j-j'|>N.$$
Then we use the following algebraic identity:
$$
\begin{array}{lll}
\bigl(\phi(2^{-j}D)(T_{v^k}\pa_k z)|\phi(2^{-j}D)z\bigr)_{L^2}&\!\!\!\!=\!\!\!\!&
\Sum_{|j'-j|\leq N} \bigl(\phi(2^{-j}D)(\dot S_{j'-1}v^k\pa_k\dot\Delta_{j'}z)|\phi(2^{-j}D)z\bigr)_{L^2}\\[1ex]
&\!\!\!\!=\!\!\!\!&\Sum_{|j'-j|\leq N}\!\bigl(\phi(2^{-j}D)((\dot S_{j'\!-\!1}\!-\!\dot S_{j\!-\!1})v^k\pa_k\dot\Delta_{j'}z)|\phi(2^{-j}D)z\bigr)_{L^2}\\
&&\hspace{1cm}+\!\Sum_{|j'-j|\leq N}\!\bigl([\phi(2^{-j}D),\dot S_{j-1}v^k]\pa_k\dot\Delta_{j'}z|\phi(2^{-j}D)z\bigr)_{L^2}\\
&&\hspace{3cm}+(\dot S_{j-1}v^k\pa_k\phi(2^{-j}D)z|\phi(2^{-j}D)z)_{L^2}.\end{array}
$$
The first term may be bounded thanks to spectral localization and Bernstein inequality, and 
the second, to a standard commutator estimate (see e.g. \cite{BCD}, Lemma 2.97).
The last term may be dealt with according to the following integration by parts:
$$
\int \dot S_{j-1}v^k\pa_k\phi(2^{-j}D)z\:\phi(2^{-j}D)z\,dx=-\frac12\int \div \dot S_{j-1}v\: (\phi(2^{-j}D)z)^2\,dx.
$$

Let us now resume to the proof of Proposition \ref{p:paralinear}.
As an example, we show how the first two steps of the previous subsection have to be adapted for \eqref{eq:lhNS}.
So we apply $\ddj$ to the first three equations and get:
$$
 \left\{
    \begin{aligned}
      & \pa_t b_j+\Lambda d_j+\ddj(T_{v^k}\pa_k b)=B_j, \\
      &\pa_t d_j + \ddj(T_{v^k}\pa_k d)- \Delta d_j-\Lambda(b_j+\theta_j)  = D_j,\\
      &\pa_t \theta_j  +\Lambda d_j +\ddj(T_{v^k}\pa_k \theta)-\tilde\kappa \Delta \theta_j=G_j.
    \end{aligned}
  \right.
$$
Taking the $L^2$-inner product of the first, second and third equations with $b_j,$ $d_j$ and $\theta_j,$
 respectively, we find that
$$\displaylines{
\f12\f{d}{dt}\big[\|b_j\|_{L^2}^2+\|d_j\|_{L^2}^2+\|\theta_j\|_{L^2}^2\big]+\|\Lambda d_j\|_{L^2}^2+\tilde\kappa\|
\Lambda \theta_j\|_{L^2}^2 +  \bigl(\ddj(T_{v^k}\pa_k b)|\ddj b\bigr)_{L^2}\hfill\cr\hfill
+\bigl(\ddj(T_{v^k}\pa_k d)|\ddj d\bigr)_{L^2}
+\bigl(\ddj(T_{v^k}\pa_k \theta)|\ddj\theta\bigr)_{L^2}=(B_j|b_j)_{L^2}+(D_j|d_j)_{L^2}+(G_j|\theta_j)_{L^2}.}
$$
 Therefore using Inequality \eqref{eq:paraconv} we readily get
 $$\displaylines{
\f12\f{d}{dt}\|(b_j,d_j,\theta_j)\|_{L^2}^2+\|\Lambda d_j\|_{L^2}^2+\tilde\kappa\|
\Lambda \theta_j\|_{L^2}^2\leq \|(b_j,d_j,\theta_j)\|_{L^2}\hfill\cr\hfill
\times \Bigl(\|(B_j,D_j,G_j)\|_{L^2}+C\|\nabla v\|_{L^\infty}\Sum_{|j'-j|\leq N}\|(b_{j'},d_{j'},\theta_{j'})\|_{L^2}\Bigr).}
$$
Next, we use the fact that $\Lambda b_j-d_j$ satisfies
$$
\pa_t(\Lambda b_j-d_j)+\Lambda(b_j+d_j)+\Lambda\ddj(T_{v^k}\pa_kb)-\ddj(T_{v^k}\pa_kd)=
\Lambda B_j-D_j.
$$
Therefore arguing as in the second step of the previous section, we get
  $$\displaylines{
  \f12\f{d}{dt}f_j^2 +H_j^2+\bigl((\Lambda\ddj(T_{v^k}\pa_kb)-\ddj(T_{v^k}\pa_kd))|(\Lambda b_j-d_j)\bigr)_{L^2}
  \hfill\cr\hfill +  (1+\alpha) \bigl(\ddj(T_{v^k}\pa_k b)|\ddj b\bigr)_{L^2}+\alpha\bigl(\ddj(T_{v^k}\pa_k d)|\ddj d\bigr)_{L^2}
+(1+\alpha)\bigl(\ddj(T_{v^k}\pa_k \theta)|\ddj\theta\bigr)_{L^2}\hfill\cr\hfill
  = (1+\alpha)(B_j|b_j)_{L^2}+\alpha(D_j|d_j)_{L^2}+(1+\alpha)(G_j|\theta_j)_{L^2}+\bigl((\Lambda B_j-D_j)|(\Lambda b_j-d_j)\bigr)_{L^2}}
$$
where $f^j$ and $H^j$ have been defined in \eqref{eq:fj} and \eqref{eq:Hj}, and  $\alpha=2/\check\kappa-1.$
\medbreak
Note that all the paraconvection terms except the first one may be directly dealt with according to \eqref{eq:paraconv}.
As for the first one,  we may use the decomposition:
$$
\Lambda\ddj(T_{v^k}\pa_kb)-\ddj(T_{v^k}\pa_kd)
=\ddj T_{v^k} \pa_k(\Lambda b-d)+2^j[\phi(2^{-j}D),T_{v^k}]\pa_k b
$$
with $\phi(\xi):=|\xi|\varphi(\xi).$
Therefore, applying again  \eqref{eq:paraconv}
and Lemma 2.97 in \cite{BCD}, we end up with
$$\displaylines{
\bigl|\bigl((\Lambda\ddj(T_{v^k}\pa_kb)-\ddj(T_{v^k}\pa_kd))|(\Lambda b_j-d_j)\bigr)_{L^2}\bigr|
\hfill\cr\hfill\lesssim \|\nabla v\|_{L^\infty}\|\Lambda b_j-d_j\|_{L^2}\!\sum_{|j'-j|\leq N}
\!\bigl(\|\Lambda b_{j'}-d_{j'}\|_{L^2}+\|\Lambda b_{j'}\|_{L^2}\bigr).}
$$

The following steps may be done similarly, once noticed that operators
such as $\Lambda^{\pm1}\ddj$ may be written $2^{\pm j}\phi(2^{-j}D)$ for some suitable
function $\phi$ with the same support as $\varphi.$
The final inequality may be obtained after multiplying  by $2^{js},$ performing
a summation over $j$ and applying Gronwall's lemma. The details are left to the reader.
\end{proof}


\subsection{The proof of global existence}\label{ss:global}

This paragraph is devoted to proving   Theorem \ref{th:main1} in the case
$\eps=\nu=1.$ As explained at the incipit of this section, this will imply
the global existence for general positive $\eps$ and $\nu.$
The proof of existence and uniqueness is similar to that for the full Navier-Stokes system in \cite{D2}.
The only difference here is that the source term $\nabla V^\eps$  is not in $L^1(\R^+;\dot B^{\f12}_{2,1}).$
However it still belongs to $L^1_{loc}(\R^+;\dot B^{\f12}_{2,1})$ which suffices to establish
local-in-time results, global results being a consequence of the following a priori estimates.
Note that a direct proof based on Friedrichs spectral truncation method may also be
easily implemented as we are interested in $L^2$ type estimates.
\medbreak
 So let us now derive  global a priori estimates under the smallness assumptions
\eqref{eq:smalldata1} and \eqref{eq:smalldata2}.
Such estimates rely on Proposition \ref{p:paralinear} with $s=1/2,$ once noticed that
$$
u=\cP u+(\Id-\cP) u= w -\nabla \Lambda^{-1} d,$$
that $(b,d,\theta,w)$ satisfies \eqref{eq:lhNS} with $v=u$ and,
 using the summation convention over repeated indices,
   \beno B&:=&T_{u^k}\pa_k b-u\cdot\na b-b\div u-\pa_t \widetilde{V}-\div( u\widetilde{V}), \\
   D&:=&T_{u^k}\pa_k d-\Lambda^{-1}\div(u\cdot\na u)
   -\Lambda^{-1}\div\bigg[\f{a}{1+ a}(\tilde\mu\Delta u+(\tilde\lambda+\tilde\mu)\nabla\div u)+ \f{ (\theta-a)\na  a }{(1+ a)}\bigg], \\
  G&:=&T_{u^k}\pa_k \theta-u\cdot\na \theta-\theta\div u-\f{a}{1+ a}\tilde{\kappa}\Delta \theta+\frac{1}{1+a}[2\tilde\mu|Du|^2+\tilde\lambda (\div u)^2],\\
W&:=& T_{u^k}\pa_k w-\cP(u\cdot\na u)
 -\cP\bigg[\f{a}{1+a}(\tilde\mu\Delta u+(\tilde\lambda+\tilde\mu)\nabla\div u)
 +\f{ (\theta-a)\na  a }{(1+ a)} \bigg]\cdotp \eeno

Setting $U(t):=\int_0^t \|\na u\|_{L^\infty}\,d\tau$ and
$$
X(t):=\|b\|_{\tilde{L}_t^\infty(\tilde B^{\f32,-}_1)}+\|u\|_{\tilde{L}_t^\infty(\dot B^\f12_{2,1})}+\|\theta\|_{\tilde{L}_t^\infty(\tilde B^{-\f12,+}_1)}
+\int_0^t\bigl(\|b\|_{\tilde B^{\f32,+}_1}+\|u\|_{\dot B^{\f52}_{2,1}}+\|\theta\|_{\tilde B^{\f32,+}_1}\bigr)\,d\tau,
$$
  we may write
\begin{multline}\label{eq:X}
X(t)\leq Ke^{CU(t)}\biggl(X(0) +\int_0^te^{-CU(\tau)}\bigl(\|B\|_{\tilde B^{\f32,-}_1}+\|(D,W)\|_{\dot B^{\f12}_{2,1}}
+\|G\|_{\tilde B^{-\f12,+}_1}\bigr)
\,d\tau\biggr)\cdotp
 \end{multline}
 Throughout we suppose that $1+a$ is bounded and bounded away from $0,$ an assumption
 that is satisfied provided $\|a\|_{L^\infty(\dot B^{\f32}_{2,1})}$ is small enough.

\subsubsection*{Bounding $\|B\|_{\tilde B^{\f12,\f32}_1}$}

According to Bony's decomposition, we have
$$
u\cdot\na b-T_{u^k}\pa_k b=T'_{\pa_k b}u^k.
$$
Hence standard results for the paraproduct imply (just decompose $b$  into low and high frequencies):
\begin{equation}\label{eq:p1}
\|T_{u^k}\pa_k b-u\cdot\na b\|_{\tilde B^{\f32,-}_1}\lesssim \|\na b\|_{\tilde B^{\f12,-}_1}\|u\|_{\dot B^{\f52}_{2,1}}.
\end{equation}
Likewise, according to Lemma \ref{est-pro}, we have
\begin{equation}\label{eq:p2}
\|b\,\div u\|_{\tilde B^{\f32,-}_1}\lesssim \|b\|_{\tilde B^{\f32,-}_1}\|\div u\|_{\dot B^{\f32}_{2,1}}.
\end{equation}
Finally, because $\div (\widetilde{V}u)=\widetilde{V}\div u+u\cdot\nabla \widetilde{V},$ we have
\begin{equation}\label{eq:p3}
\|\div (u \widetilde{V})\|_{\tilde B^{\f32,-}_1}\lesssim
 \|\widetilde{V}\|_{\tilde B^{\f32,-}_1}\|\div u\|_{\dot B^{\f32}_{2,1}}+
 \|\nabla \widetilde{V}\|_{\tilde B^{\f32,-}_1}\|u\|_{\dot B^{\f32}_{2,1}}.
\end{equation}

\subsubsection*{Bounding $\|(D,W)\|_{\dot B^{\f12}_{2,1}}$}

We concentrate on $D,$ proving estimates for $W$ being similar.
We have
$$
T_{u^k}\pa_k d-\Lambda^{-1}\div(u\cdot\na u)=[T_{u^k},\Lambda^{-1}\pa_i]\pa_ku^i-\Lambda^{-1}\pa_iT'_{\pa_ku^i}u^k.
$$
Therefore, resorting to standard commutator estimates and continuity results for the paraproduct
(see e.g. \cite{BCD}), we get
\begin{equation}\label{eq:p4}
\|T_{u^k}\pa_k d-\Lambda^{-1}\div(u\cdot\na u)\|_{\dot B^{\f12}_{2,1}}\lesssim
\|\nabla u\|_{L^\infty}\|u\|_{\dot B^{\f12}_{2,1}}.
\end{equation}
Next, combining composition and product estimates yields
\begin{equation}\label{eq:p5}
\|\f{a}{1+ a}\nabla^2u\|_{\dot B^{\f12}_{2,1}}\lesssim \|a\|_{\dot B^{\f32}_{2,1}}\|u\|_{\dot B^{\f52}_{2,1}},
\end{equation}
and also
$$
\Bigl\|\f{(\theta-a)\na  a }{1+ a}\Bigr\|_{\dot B^{\f12}_{2,1}}\lesssim \bigl(1+ \|a\|_{\dot B^{\f32}_{2,1}}\bigr)
\|\nabla a\|_{\dot B^{\f12}_{2,1}}\bigl(\|\theta^\ell\|_{\dot B^{\f32}_{2,1}}+\|\theta^h\|_{\dot B^{\f32}_{2,1}}
+\|a\|_{\dot B^{\f32}_{2,1}}\bigr).
$$
Note that we expect $\theta^\ell$ and $\theta^h$ to belong to
$L^2(\R^+;\dot B^{\f32}_{2,1})$ and $L^1(\R^+;\dot B^{\f32}_{2,1}),$ respectively, and that, applying H\"older inequality yields
$$
\Bigl\|\f{(\theta\!-\!a)\na  a }{1+ a}\Bigr\|_{L^1(\dot B^{\f12}_{2,1})}\lesssim \bigl(1+ \|a\|_{L^\infty(\dot B^{\f32}_{2,1})}\bigr)
\bigl(\|a\|_{L^2(\dot B^{\f32}_{2,1})}\|(a,\theta^\ell)\|_{L^2(\dot B^{\f32}_{2,1})}
+\|a\|_{L^\infty(\dot B^{\f32}_{2,1})}\|\theta^h\|_{L^1(\dot B^{\f32}_{2,1})}\bigr).
$$
So finally, because $\tilde B^{\f32,-}_1\hookrightarrow\dot B^{\f32}_{2,1},$
\begin{multline}\label{eq:p6}
\Bigl\|\f{(\theta-a)\na  a }{1+ a}\Bigr\|_{L^1(\dot B^{\f12}_{2,1})}\lesssim \bigl(1+ \|a\|_{L^\infty(\tilde B^{\f32,-}_1)}\bigr)\\\times
\bigl(\|a\|_{L^2(\dot B^{\f32}_{2,1})}^2+\|a\|_{L^2(\dot B^{\f32}_{2,1})}\|\theta\|_{L^2(\tilde B^{\f12,+}_{1})}
+\|a\|_{L^\infty(\tilde B^{\f32,-}_{1})}\|\theta\|_{L^1(\tilde B^{\f32,+}_{1})}\bigr).
\end{multline}

\subsubsection*{Bounding $\|G\|_{\tilde B^{\f12,+}_1}$}

We first use the fact that
$$u\cdot\na \theta-
T_{u^k}\pa_k \theta=T'_{\pa_k\theta}u^k.
$$
Therefore
\begin{equation}\label{eq:p7}
\|T_{u^k}\pa_k \theta-u\cdot\na \theta\|_{\tilde B^{-\f12,+}_1}\lesssim \|\nabla\theta\|_{\tilde B^{-\f32,+}_1}
\|u\|_{\dot B^{\f52}_{2,1}}.
\end{equation}
Next, Lemma \ref{est-pro} implies that
\begin{eqnarray}\label{eq:p8}
&&\|\theta\,\div u\|_{\tilde B^{-\f12,+}_1}\lesssim \|\theta\|_{\tilde B^{-\f12,+}_1}\|\div u\|_{\dot B^{\f32}_{2,1}},\\\label{eq:p9}
&&\|\f{a}{1+ a}\Delta \theta\|_{\tilde B^{-\f12,+}_1}\lesssim \|a\|_{\dot B^{\f32}_{2,1}}\| \theta\|_{\tilde B^{\f32,+}_1}.
\end{eqnarray}
Finally, since $\dot B^{-\f12}_{2,1}\hookrightarrow\tilde B^{-\f12,+}_1,$ standard product laws
 enable us to write that
\begin{equation}\label{eq:p10}
 \Bigl\|\frac{1}{1+a}\nabla u\otimes \nabla u\Bigr\|_{\tilde B^{-\f12,+}_1}\lesssim \bigl(1+ \|a\|_{\dot B^{\f32}_{2,1}}\bigr)
 \|\nabla u\|_{\dot B^{1/2}_{2,1}}^2.
\end{equation}
 Plugging inequalities \eqref{eq:p1} to \eqref{eq:p10} in \eqref{eq:X} and making the assumption that
 \begin{equation}\label{eq:small}
 \|\nabla u\|_{L^1(L^\infty)}\ll 1\quad\hbox{and}\quad
 \|\widetilde{V}\|_{L^\infty(\tilde  B^{\f32,-}_{1})}
 +\|\nabla \widetilde{V}\|_{L^2(\tilde B^{\f32,-}_1)}\ll1,
 \end{equation}
 we thus get
 $$
 X(t)\leq C\bigl(X(0)+\|\pa_t\tilde V\|_{L^1(\tilde B^{\f32,-}_1)}  +X^2(t)+X^4(t)\bigr).
 $$
 It is now clear that the solution may be bounded for all time if $X(0)$ and $\tilde V$ are  small enough:
 we get for some constant $K$ depending only on $\tilde\kappa,$ $\tilde\mu$ and $\tilde\lambda,$
 \begin{equation}\label{eq:boundX}
 X(t)\leq KC_0
 \end{equation}
 with
 $$
 C_0:=\|b_0\|_{\tilde B^{\f32,-}_1}+\|u_0\|_{\dot B^{\f12}_{2,1}}+\|\theta_0\|_{\tilde B^{-\f12,+}_1}+
 \|\pa_t\widetilde{V}\|_{L^1(\tilde B^{\f32,-}_1)}.
$$


\subsection{Convergence to the viscous and diffusive Boussinesq system}

The key observation is that in the asymptotics $\eps$ going to $0,$
the leading order part of the system for $(q^\eps,\cQ u^\eps)$ is the acoustic wave equation,
which has dispersive properties. This will enable us to show (first step)
that  $(q^\eps,\cQ u^\eps)$ tends strongly to $0$ in some negative Besov space.
Next, we shall check that the limit Boussinesq system \eqref{eq:BouS}  supplemented with
small data $v_0\in \dot B^{\f12}_{2,1},$ $\Theta_0\in\dot B^{-\f12}_{2,1}$ and
potential $V$ with $\pa_tV\in L^1(\dot B^{\f12}_{2,1})$ and $\nabla V\in L^2(\dot B^{\f12}_{2,1})$  has
a unique global solution.  Finally, resorting to
maximal regularity estimates for the heat equation, we will conclude that
$(\cP u^\eps,\Theta^\eps)\to (v,\Theta).$

\subsubsection{Convergence to zero for the oscillating modes $(q^\eps,\cQ u^\eps)$}
In order to exhibit the decay properties of $(q^\eps,\cQ u^\eps),$  we only have to consider
the case $\eps=1$ and $\nu=1$ thanks to the rescaling  \eqref{eq:change}, which implies
in particular that
$$
(q,\cQ u)(t,x)=\eps(q^\eps,\cQ u^\eps)(\eps^2\nu t,\eps\nu x).
$$
Then  using Strichartz estimates for the acoustic wave equation (see Proposition \ref{p:strichartz}
in the appendix)
will enable us to bound some suitable norm of $(q,\cQ u).$ Resuming to the
original variables, we then get for free the convergence to $0$
 for $(q^\eps,\cQ u^\eps)$, with an explicit rate.
\medbreak
Let us give more details : $(q,\cQ u)$ satisfies
\begin{equation}\label{eq:oscbis}
  \left\{
    \begin{aligned}
      & \pa_t q +\sqrt2\div \cQ u=-\div(q u)
      -\frac{\sqrt2}2\biggl(\pa_t\widetilde{V}+\div(\widetilde{V}u)+\tilde\kappa\frac{\Delta\theta}{1+ a}\biggr)\\
      &\hspace{7cm}+\frac{\sqrt2}2\frac{1}{1+ a}[2\mu|Du|^2+\lambda (\div u)^2],\\
      &\pa_t \cQ u + \sqrt2\nabla q=\cQ\biggl(\biggl(\frac{a-\theta}{1+a}\biggr)\nabla a-
      \f{\cA u}{1+a} -u\cdot\nabla u\biggr)\cdotp
    \end{aligned}
  \right.
\end{equation}
Therefore Strichartz estimates (first inequality of Proposition \ref{p:strichartz} with $s=1/2$)
enable us to bound the norm of $(q,\cQ u)$ in $\tilde L^{\f{2p}{p-2}}(\dot B^{\f2p-\f12}_{p,1})$ for all $p\in[2,\infty)$
in terms of the norm of the data in $\dot B^{\f12}_{2,1}$ and of the right-hand side in $L^1(\dot B^{\f12}_{2,1}).$
Under  our present assumptions however,  the last term in the r.h.s. of the first equation
belongs only to the larger space
 $L^1(\tilde B^{-\f12,+}_{1}).$
So  one has to use the \emph{second} inequality of Proposition \ref{p:strichartz} and  just get estimates in
the wider space $\tilde L^{\f{2p}{p-2}}(\tilde B^{\f2p-\f32,+}_{p,1}).$

Let us bound the r.h.s. of \eqref{eq:oscbis} in  $L^1(\tilde B^{-\f12,+}_{1}).$
All the terms may be dealt with by taking advantage of standard product laws and  Lemma \ref{est-pro}.
More precisely we have, keeping in mind the smallness of $a$ in $L^\infty(\tilde B^{\f32,-}_{1})$
(and thus also in $L^\infty(\dot B^{\f32}_{2,1})$ and $L^\infty(\R^+\times\R^3)$):
$$\!\!
\begin{array}{lll}
\|\div(\widetilde{V}u)\|_{L^1(\dot B^{\f12}_{2,1})}&\!\!\!\!\lesssim\!\!\!\!&\|\widetilde{V}\|_{L^2(\dot B^{\f32}_{2,1})}\|u\|_{L^2(\dot B^{\f32}_{2,1})},\\[1ex]
\|(1\!+\!a)^{-1}\Delta\theta\|_{L^1(\tilde B^{-\f12,+}_{1})}&\!\!\!\!\lesssim\!\!\!\!&\|\theta\|_{L^1(\tilde B^{\f32,+}_{1})},\\[1ex]
    \|(1\!+\!a)^{-1}\nabla u\!\otimes\! \nabla u\|_{L^1(\dot B^{-\f12}_{2,1})}&\!\!\!\!\lesssim\!\!\!\!&\|u\|^2_{L^2(\dot B^{\f32}_{2,1})},\\[1ex]
\|(1\!+\!a)^{-1}(a\!-\!\theta)\nabla a\|_{L^1(\dot B^{\f12}_{2,1})}&\!\!\!\!\lesssim\!\!\!\!&\|a\|_{L^2(\dot B^{\f32}_{2,1})}
(\|a\|_{L^2(\dot B^{\f32}_{2,1})}\!\!+\!\!\|\theta\|_{L^2(\tilde B^{\f12,+}_1)})\!+\!\|a\|_{L^\infty(\tilde B^{\f32,-}_{1})}
\|\theta\|_{L^1(\tilde B^{\f32,+}_{1})}\\[1ex]
\|(1\!+\!a)^{-1}\cA u\|_{L^1(\dot B^{\f12}_{2,1})}&\!\!\!\!\lesssim\!\!\!\!&\|u\|_{L^1(\dot B^{\f52}_{2,1})},\\[1ex]
\|u\cdot\nabla u\|_{L^1(\dot B^{\f12}_{2,1})}&\!\!\!\!\lesssim\!\!\!\!&\|u\|^2_{L^2(\dot B^{\f32}_{2,1})}.
\end{array}\!\!
 $$
Given that $\cQ$ is a $0$-th order multiplier (hence maps
all Besov spaces involved here into themselves), that
$\dot B^{-\f12}_{2,1}$ and $\dot B^{\f12}_{2,1}$ are continuously embedded in $\tilde B^{-\f12,+}_1,$
and that $\tilde B^{\f32,-}_1$ is continuously embedded in $\tilde B^{-\f12,+}_1,$
we eventually conclude that (with the notation of \eqref{eq:X} and \eqref{eq:boundX}):
$$
\|(q,\cQ u)\|_{\tilde L^{\f{2p}{p-2}}(\tilde B^{\f2p-\f32,+}_{p,1})}\lesssim
\|(q_0,\cQ u_0)\|_{\tilde B^{-\f12,+}_1}+X+X^2+\|\pa_t\tilde V\|_{L^1(\tilde B^{\f32,-}_1)}+
\|\tilde V\|_{L^2(\dot B^{\f32}_{2,1})}X.
$$
Therefore, given that $X(t)\leq KC_0$ and that $C_0$ is small,
\begin{equation}\label{eq:comp1}
\|(q,\cQ u)\|_{\tilde L^{\f{2p}{p-2}}(\tilde B^{\f2p-\f32,+}_{p,1})}\leq KC_0\quad\hbox{for all }\ p\in[2,\infty)
\end{equation}
with $K$ depending only on $p,$ $\tilde\kappa$ and  $\tilde\mu.$
\medbreak
On the other hand,  Inequality \eqref{eq:boundX} implies that
$$
\|(q,\cQ u)\|_{L^1(\tilde B^{\f32,+}_1)}\leq KC_0.
$$
Therefore, using the fact that
$$
[L^1(\tilde B^{\f32,+}_1),\tilde L^{\f{2p}{p-2}}(\tilde B^{\f2p-\f32,+}_{p,1})]_{p/(p+2)}
\subset \wt L^2(\tilde B^{\f4q-\f32}_{q,1})\quad\hbox{with } q=(p+2)/2,
$$
we get also
$$
\|(q,\cQ u)\|_{\wt L^2(\tilde B^{\f4q-\f32,+}_{q,1})}\leq KC_0\quad\hbox{for all }\ q\in[2,\infty).
$$
Given that $(q,\cQ u)$ is in $\tilde{L}^2(\tilde B^{\f12,+}_{1})$ hence in $\tilde{L}^2(\tilde B^{-1+\f3q,+}_{q,1}),$
an ultimate interpolation  ensures that
\begin{equation}\label{eq:comp2}
\|(q,\cQ u)\|_{\wt L^2(\tilde B^{s-1,+}_{p,1})}\leq KC_0\quad\hbox{for all}\quad
s\in [-1/2+4/p,3/p] \ \hbox{ and }\ p\in[2,\infty).
\end{equation}
Of course, we also have $\cQ u$ in $\wt L^2(\dot B^{\f32}_{2,1})$ whence in $\wt L^2(\dot B^{\f3p}_{p,1})$
for all $p\geq2.$ Therefore, interpolating with \eqref{eq:comp2},
we deduce that
\begin{equation}\label{eq:comp2a}
\|\cQ u\|_{\wt L^2(\dot B^{s}_{p,1})}\leq KC_0\quad\hbox{for all}\quad s\in [-1/2+4/p,3/p].
\end{equation}
Now coming back to the initial variables, \eqref{eq:comp1}, \eqref{eq:comp2}
 and \eqref{eq:comp2a} translate  into
\begin{eqnarray}\label{eq:comp3}
&&\nu^{\f12-\f1p}\|(q^\eps,\cQ u^\eps)\|_{\tilde L^{\f{2p}{p-2}}(\tilde B^{\f2p-\f32,+}_{p,\eps\nu})}
\leq K(\eps\nu)^{\f12-\f1p}C_0^\eps\quad\hbox{for all }\ p\in[2,\infty),\\ \label{eq:comp4}
&&\nu^{\f12}\|q^\eps\|_{\tilde L^2(\tilde B^{s-1,+}_{p,\eps\nu})}\leq K(\eps\nu)^{\f3p-s}C_0^\eps\ \hbox{ for all }\  s\in [-1/2+4/p,3/p],\\ \label{eq:comp5}
&&\nu^{\f12}\|\cQ u^\eps\|_{\tilde L^2(\dot B^{s}_{p,1})}\leq K(\eps\nu)^{\f3p-s}C_0^\eps
\ \hbox{ for all }\  s\in [-1/2+4/p,3/p].
\end{eqnarray}


\subsubsection{Global existence for the Boussinesq system \eqref{eq:BouS}}

Let us first briefly justify that, under our assumptions, the limit data $(\Theta_0,v_0,V)$ give
rise to a global solution to \eqref{eq:BouS}.
Establishing this  is an obvious modification of the proof for the standard incompressible
Navier-Stokes equation. It is only a matter of rewriting the system as
$$\begin{aligned}
\Theta(t)&=e^{t\frac\kappa2\Delta}\Theta_0+\int_0^te^{(t-\tau)\f\kappa2\Delta}\Bigl(
\f{\sqrt2}2(\pa_tV+v\cdot\nabla V)-v\cdot\nabla\Theta\Bigr)\,d\tau,\\
v(t)&=e^{t\mu\Delta}v_0-\int_0^te^{\mu(t-\tau)\Delta}\cP\Bigl(v\cdot\nabla v+\f{\sqrt2}2\Theta\nabla V\Bigr)\,d\tau,
\end{aligned}
$$
and the global-in-time solvability for small data may be achieved as a consequence of
the Banach fixed point theorem.
Let us just check that global a priori estimates are available in the case of small data.
Applying Proposition \ref{p:heatestimates}  and using  that the product is continuous from
$\dot B^{\f12}_{2,1}\times\dot B^{\f32}_{2,1}$ to $\dot B^{\f12}_{2,1}$ implies that
$$
\displaylines{\quad\|\Theta\|_{\tilde L^\infty(\dot B^{\f12}_{2,1})}+\kappa\|\Theta\|_{L^1(\dot B^{\f52}_{2,1})}
\lesssim\|\Theta_0\|_{\dot B^{\f12}_{2,1}}+\|\pa_tV\|_{L^1(\dot B^{\f12}_{2,1})}
\hfill\cr\hfill+\|v\|_{L^2(\dot B^{\f32}_{2,1})}\bigl(\|\nabla\Theta\|_{L^2(\dot B^{\f12}_{2,1})}+\|\na V\|_{L^2(\dot B^{\f12}_{2,1})}\bigr)\quad}
$$
and that
$$
\|v\|_{\tilde L^\infty(\dot B^{\f12}_{2,1})}+\mu\|v\|_{L^1(\dot B^{\f52}_{2,1})}
\lesssim\|v_0\|_{\dot B^{\f12}_{2,1}}
+\|v\|_{L^2(\dot B^{\f32}_{2,1})}\|\nabla v\|_{L^2(\dot B^{\f12}_{2,1})}+
\|\na V\|_{L^2(\dot B^{\f12}_{2,1})}\|\Theta\|_{L^2(\dot B^{\f32}_{2,1})}.
$$
Hence, setting
$$
Y:= \|(\Theta,v)\|_{\tilde L^\infty(\dot B^{\f12}_{2,1})}+\nu\|(\Theta,v)\|_{L^1(\dot B^{\f52}_{2,1})},
$$
we get for some constant $K=K(\tilde\mu,\tilde\kappa),$
$$
Y\leq K\bigl(Y_0+\|\pa_tV\|_{L^1(\dot B^{\f12}_{2,1})}+Y(Y+\nu^{-\f12}\|\na V\|_{L^2(\dot B^{\f12}_{2,1})})\bigr),
$$
and it thus easy to close the estimates globally if $Y_0,$  $\|\pa_tV\|_{L^1(\dot B^{\f12}_{2,1})}$
and  $\nu^{\f12}\|\na V\|_{L^2(\dot B^{\f12}_{2,1})}$ are small compared to $\nu.$


\subsubsection{Convergence for the ``incompressible'' modes $(\Theta^\eps,\cP u^\eps)$}

In this paragraph, we prove the convergence of $(\Theta^\eps,\cP u^\eps)$ to the solution
$(\Theta,v)$ to the Boussinesq equation \eqref{eq:BouS}.
 We  claim that for any $p\in[2,\infty)$ and $s\in[-1/2+4/p,3/p]$ with $s>1/2$~:
\begin{itemize}
\item $\dT:=\Theta^\eps-\Theta$  tends to $0$ in
$\tilde L^2(\tilde B^{s-1,+}_{p,\eps\nu})\cap \tilde L^\infty(\tilde B^{s-2,+}_{p,\eps\nu}),$
\item $\dv:=\cP u^\eps-v$  tends to $0$ in
$L^1(\tilde B^{s,+}_{p,\eps\nu})\cap \tilde L^\infty(\tilde B^{s-2,+}_{p,\eps\nu}).$
\end{itemize}
For proving that, we shall use the parabolic estimates of Proposition \ref{p:heatestimates} for
the system satisfied by $(\dT,\dv).$
Let us first focus on $\dT.$ By performing the difference between \eqref{eq:nonosc} and \eqref{eq:BouS},
 we see that
$$
\displaylines{\quad
\pa_t\dT
-\f\kappa2\Delta\dT=-\cP u^\eps\cdot\nabla\dT-\dv\cdot\nabla\Theta+\f{\sqrt2}2\bigl(\pa_t\dV+\cP u^\eps\!\cdot\!\nabla\dV+\dv\cdot\nabla V\bigr)\hfill\cr\hfill
+\div((V^\eps-\Theta^\eps)\cQ u^\eps)
+\f\kappa2\Delta q^\eps-\f{\sqrt2}2\kappa\frac{\eps a^\eps}{1+\eps a^\eps}\Delta\theta^\eps
      +\f{\sqrt2}2\frac{\eps}{1+\eps a^\eps}[2\mu|Du^\eps|^2\!+\!\lambda (\div u^\eps)^2].\quad}
$$
Hence, according to Proposition \ref{p:heatestimates},  it suffices to get suitable estimates for the right-hand side in
$L^1(\tilde B^{s-2,+}_{p,\eps\nu})+ \tilde L^2(\tilde B^{s-3,+}_{p,\eps\nu}).$
{}From product estimates (see Lemma \ref{est-pro}) we easily get under the assumption that $s>1/2$
(in fact here we just need $s>-1/2$ owing to $\div\dv=\div\cP u^\eps=0$):
\begin{eqnarray}\label{eq:theta1}
&&\|\cP u^\eps\cdot\nabla\dT\| _{L^1(\tilde B^{s-2,+}_{p,\eps\nu})}\lesssim
\|\cP u^\eps\|_{L^2(\dot B^{\f32}_{2,1})}\|\nabla\dT\|_{L^2(\tilde B^{s-2,+}_{p,\eps\nu})},\\ \label{eq:theta2}
&&\|\dv\cdot\nabla\Theta\|_{L^1(\tilde B^{s-2,+}_{p,\eps\nu})}\lesssim
\|\nabla\Theta\|_{L^2(\dot B^{\f12}_{2,1})}\|\dv\|_{L^2(\tilde B^{s-1,+}_{p,\eps\nu})},\\ \label{eq:theta3}
&&\|\cP u^\eps\!\cdot\!\nabla\dV\|_{L^1(\tilde B^{s-2,+}_{p,\eps\nu})}\lesssim
\|\cP u^\eps\|_{L^2(\dot B^{\f32}_{2,1})}\|\nabla\dV\|_{L^2(\tilde B^{s-2,+}_{p,\eps\nu})},\\ \label{eq:theta4}
&&\|\dv\cdot\nabla V\|_{L^1(\tilde B^{s-2,+}_{p,\eps\nu})}\lesssim
\|\nabla V\|_{L^2(\dot B^{\f12}_{2,1})}\|\dv\|_{L^2(\tilde B^{s-1,+}_{p,\eps\nu})}.
\end{eqnarray}
We split the next term into (referring to the notation introduced in \eqref{eq:decompo} with $\alpha=\eps\nu$)
$$
\div((V^\eps-\Theta^\eps)\cQ u^\eps)=\div((V^\eps-\Theta^{\eps,\ell})\cQ u^\eps)
-\div(\Theta^{\eps,h}\cQ u^\eps).
$$
First we have
\begin{align}\label{eq:theta5}
\|\div((V^\eps-\Theta^{\eps,\ell})\cQ u^\eps)\|_{L^1(\tilde B^{s-2,+}_{p,\eps\nu})}
&\lesssim \|(V^\eps-\Theta^{\eps,\ell})\cQ u^\eps\|_{L^1(\tilde B^{s-1,+}_{p,\eps\nu})}\nonumber\\
&\lesssim (\|V^\eps\|_{L^2(\dot B^{\f32}_{2,1})}+\|\Theta^{\eps,\ell}\|_{L^2(\dot B^{\f32}_{2,1})})
\|\cQ u^\eps\|_{L^2(\tilde B^{s-1,+}_{p,\eps\nu})},
\end{align}
and, second
\begin{equation}\label{eq:theta6}
\|\div(\Theta^{\eps,h}\cQ u^\eps)\|_{L^1(\tilde B^{s-2,+}_{p,\eps\nu})}\lesssim \f1{\epsilon\nu}
\|\Theta^{\eps,h}\cQ u^\eps\|_{L^1(\dot B^{s-1}_{p,1})}\lesssim
\|\cQ u^\eps\|_{L^2(\dot B^{s}_{p,1})}\|\Theta^{\eps,h}\|_{L^2(\tilde B^{\f12,+}_{\eps\nu})}.
\end{equation}
Next, we see that, for all $\alpha\in[0,1),$
$$
\|\frac{\eps a^\eps}{1+\eps a^\eps}\Delta\theta^\eps\|_{L^1(\tilde B^{-\f12-\alpha,+}_{\eps\nu})}
\lesssim\|\eps a^\eps\|_{L^\infty(\dot B^{\f32-\alpha}_{2,1})}\|\Delta\theta^\eps\|_{L^1(\tilde B^{-\f12,+}_{\eps\nu})}.
$$
Now, by interpolation
$$
\| a^\eps\|_{\dot B^{\f32-\alpha}_{2,1}}\lesssim \|a^\eps\|_{\dot B^{\f32}_{2,1}}^{1-\alpha} \|a^\eps\|_{\dot B^{\f12}_{2,1}}^{\alpha}
$$
and the definition of the norm in $\tilde B^{\f32,-}_{\eps\nu}$ implies that
$$
\|a^\eps\|_{\dot B^{\f12}_{2,1}}+\eps\nu\|a^\eps\|_{\dot B^{\f32}_{2,1}} \lesssim\|a^\eps\|_{\tilde B^{\f32,-}_{2,\eps\nu}}.
$$
Therefore
\begin{equation}\label{eq:alpha}
\|\eps\nu a^\eps\|_{\dot B^{\f32-\alpha}_{2,1}}\lesssim(\eps\nu)^\alpha\|a^\eps\|_{\tilde B^{\f32,-}_{2,\eps\nu}}.
\end{equation}
We also notice that $\tilde B^{-\f12-\alpha,+}_{2,\eps\nu}\hookrightarrow \tilde B^{\f3p-2-\alpha,+}_{p,\eps\nu}$ for $p\geq2.$
Therefore if we take
$$
 \alpha:=3/p-s,
 $$ then we get, keeping in mind that $\|a^\eps\|_{L^\infty(\tilde B^{\f32,-}_{2,\eps\nu})}$ is small,
\begin{equation}\label{eq:theta7}
\|\frac{\eps a^\eps}{1+\eps a^\eps}\Delta\theta^\eps\|_{L^1(\tilde B^{s-2,+}_{p,\eps\nu})}
\lesssim \nu^{-1}(\eps\nu)^{\alpha}\|a^\eps\|_{L^\infty(\tilde B^{\f32,-}_{2,\eps\nu})}
\|\theta^\eps\|_{L^1(\tilde B^{\f32,+}_{2,\eps\nu})}.
\end{equation}
Finally,
$$
\begin{aligned}
\|\frac{\eps}{1+\eps a^\eps}[2\mu|Du^\eps|^2\!+\!\lambda (\div u^\eps)^2]\|_{L^1(\dot B^{-\f12}_{2,1})}
&\lesssim\eps (1+\|\eps a^\eps\|_{L^\infty(\dot B^{\f32}_{2,1})})\|\nabla u^\eps\|_{L^2(\dot B^{\f12}_{2,1})}^2,\\
 &\lesssim \eps(1+\nu^{-1}\|a^\eps\|_{L^\infty(\tilde B^{\f32,-}_{2,\eps\nu})})\|u^\eps\|_{L^2(\dot B^{\f32}_{2,1})}^2.
\end{aligned}
$$
At this point, let us notice that for all $z\in\dot B^{-\f12}_{2,1}$ and $\alpha\in[0,1],$
$$\begin{array}{lll}
\|z\|_{\dot B^{-\f12-\alpha,+}_{2,\eps\nu}}&=&\|z^\ell\|_{\dot B^{\f12-\alpha}_{2,1}}+(\eps\nu)^{-1}\|z^h\|_{\dot B^{-\f12-\alpha}_{2,1}}\\&\lesssim&(\eps\nu)^{\alpha-1}\|z\|_{\dot B^{-\f12}_{2,1}}.
\end{array}
$$
Since  $\tilde B^{-\f12-\alpha,+}_{2,\eps\nu}\hookrightarrow \tilde B^{s-2,+}_{p,\eps\nu}$
(with $\alpha=3/p-s$), we thus end up with
\begin{equation}\label{eq:theta8}
\|\frac{\eps}{1+\eps a^\eps}[2\mu|Du^\eps|^2\!+\!\lambda (\div u^\eps)^2]\|_{L^1(\tilde B^{s-2,+}_{p,\eps\nu})}
\lesssim\nu^{-1}(\eps\nu)^\alpha \|u^\eps\|_{L^2(\dot B^{\f32}_{2,1})}^2.
\end{equation}
So putting \eqref{eq:theta1} to \eqref{eq:theta8} together and using \eqref{eq:unif},  we conclude that
\begin{eqnarray}\label{eq:theta9}
&&\nu^{\f12}\|\dT\|_{\tilde L^2(\tilde B^{s-1,+}_{p,\eps\nu})}+\|\dT\|_{\tilde L^\infty(\tilde B^{s-2,+}_{p,\eps\nu})}
\lesssim \|\dT_0\|_{\tilde B^{s-2,+}_{p,\eps\nu}}+M_0\|(\dv,\dT)\||_{\tilde L^2(\tilde B^{s-1,+}_{p,\eps\nu})}\nonumber\\
&&\qquad\qquad+M_0(\eps\nu)^{\alpha}(M_0+1)+M_0\|\nabla\dV\|_{L^2(\tilde B^{s-2,+}_{p,\eps\nu})}
+\|\pa_t\dV\|_{L^1(\tilde B^{s-2,+}_{p,\eps\nu})+\tilde L^2(\tilde B^{s-3,+}_{p,\eps\nu})}.
\end{eqnarray}
Let us now concentrate on the proof of estimates for $\dv.$ We have,
subtracting \eqref{eq:BouS} from \eqref{eq:nonosc} and using \eqref{eq:relation},
$$
\displaylines{\quad
      \!\pa_t\dv-\mu\Delta\dv +\cP(\cP u^\eps\!\cdot\!\nabla\dv+\dv\!\cdot\!\nabla v)
    \!=\!-\frac{\sqrt2}2\cP\bigl(\Theta^\eps\nabla\dV\!+\!\dT\nabla V\!+q^\eps\nabla V^\eps
    - 2q^\eps\nabla b^\eps\bigr)\hfill\cr\hfill
    -\cP\biggl(  u^\eps\cdot\nabla\cQ u^\eps+\cQ u^\eps\cdot\nabla\cP u^\eps+\frac{\eps a^\eps}{1+\eps a^\eps}\cA u^\eps
      -\frac{\eps a^\eps(\theta^\eps-a^\eps)}{1+\eps a^\eps}\nabla a^\eps\biggr)\cdotp\quad}
$$
Therefore, according to Proposition \ref{p:heatestimates} and to the fact that $\cP$ is a self-map on any homogeneous
Besov space, we have
$$
\displaylines{\|\dv\|_{\tilde L^\infty(\tilde B^{s-2,+}_{p,\eps\nu})}+\nu\|\dv\|_{L^1(\tilde B^{s,+}_{p,\eps\nu})}
\lesssim \|\dv_0\|_{\tilde B^{s-2,+}_{p,\eps\nu}}+\|\cP u^\eps\cdot\nabla\dv\|_{L^1(\tilde B^{s-2,+}_{p,\eps\nu})}
\hfill\cr\hfill+\|\dv\cdot\nabla v\||_{L^1(\tilde B^{s-2,+}_{p,\eps\nu})}
+\|\Theta^\eps\nabla\dV\|_{L^1(\tilde B^{s-2,+}_{p,\eps\nu})}+\|\dT\nabla V\|_{L^1(\tilde B^{s-2,+}_{p,\eps\nu})}
 \hfill\cr\hfill+\|q^\eps\nabla V^\eps\|_{L^1(\tilde B^{s-2,+}_{p,\eps\nu})}
   +\|q^\eps\nabla b^\eps\|_{L^1(\tilde B^{s-2,+}_{p,\eps\nu})}
    +\| u^\eps\cdot\nabla\cQ u^\eps\|_{L^1(\tilde B^{s-2,+}_{p,\eps\nu})}
    \hfill\cr\hfill+\|\cQ u^\eps\cdot\nabla\cP u^\eps\|_{L^1(\tilde B^{s-2,+}_{p,\eps\nu})}
        +\|\frac{\eps a^\eps}{1+\eps a^\eps}\cA u^\eps\|_{L^1(\tilde B^{s-2,+}_{p,\eps\nu})}
     +\|\frac{\eps a^\eps(\theta^\eps-a^\eps)}{1+\eps a^\eps}\nabla a^\eps\|_{L^1(\tilde B^{s-2,+}_{p,\eps\nu})}.}
$$
The following inequalities stem from product laws (see Lemma \ref{est-pro}), under the assumption that $s>-1/2$:
\begin{align}\label{eq:v1}
 \|\cP u^\eps\!\cdot\!\nabla\dv\|_{L^1(\tilde B^{s-2,+}_{p,\eps\nu})}&\lesssim
 \|\cP u^\eps\|_{L^2(\dot B^{\f32}_{2,1})}\|\nabla\dv\|_{L^2(\tilde B^{s-2,+}_{p,\eps\nu})}, \\\label{eq:v2}
 \|\dv\cdot\nabla v\||_{L^1(\tilde B^{s-2,+}_{p,\eps\nu})}&\lesssim \|\nabla v\|_{L^2(\dot B^{\f12}_{2,1})} \|\dv\|_{L^2(\tilde B^{s-1,+}_{p,\eps\nu})}.
\end{align}
Next we have, if $s>1/2,$
\begin{align}
   \label{eq:v3}
   \| u^\eps\cdot\nabla\cQ u^\eps\|_{L^1(\tilde B^{s-2,+}_{p,\eps\nu})} &\lesssim
    \| u^\eps\|_{L^2(\dot B^{\f32}_{2,1})} \|\nabla\cQ u^\eps\|_{L^2(\tilde B^{s-2,+}_{p,\eps\nu})},
     \\\label{eq:v4}
 \|\cQ u^\eps\cdot\nabla\cP u^\eps\|_{L^1(\dot B^{s-1}_{p,1})}&\lesssim
  \|\nabla\cP u^\eps\|_{L^2(\dot B^{\f12}_{2,1})}
  \|\cQ u^\eps\|_{L^2(\dot B^{s}_{p,1})},\\\label{eq:v5}
\|\Theta^\eps\nabla\dV\|_{L^1(\tilde B^{s-2,+}_{p,\eps\nu})}&\lesssim \|\Theta^\eps\|_{L^2(\tilde B^{\f12,+}_{\eps\nu})}
\|\nabla\dV\|_{L^2(\dot B^{s-1}_{p,1})},\\\label{eq:v6}
\|\dT\nabla V\|_{L^1(\tilde B^{s-2,+}_{p,\eps\nu})}&\lesssim \|\nabla V\|_{L^2(\dot B^{\f12}_{2,1})}\|\dT\|_{L^2(\tilde B^{s-1,+}_{p,\eps\nu})},\\\label{eq:v7}
\|q^\eps\nabla V^\eps\|_{L^1(\tilde B^{s-2,+}_{p,\eps\nu})}&\lesssim  \|\nabla V^\eps\|_{L^2(\dot B^{\f12}_{2,1})}\|q^\eps\|_{L^2(\tilde B^{s-1,+}_{p,\eps\nu})},
\\\label{eq:v8}
\|q^\eps\nabla b^\eps\|_{L^1(\tilde B^{s-2,+}_{p,\eps\nu})}&\lesssim  \|\nabla b^\eps\|_{L^2(\dot B^{\f12}_{2,1})}\|q^\eps\|_{L^2(\tilde B^{s-1,+}_{p,\eps\nu})}.
\end{align}
So arguing as in the proof of \eqref{eq:theta7}, we get
\begin{equation}\label{eq:v9}
\|\frac{\eps a^\eps}{1+\eps a^\eps}\cA u^\eps\|_{L^1(\tilde B^{s-2,+}_{p,\eps\nu})}
\lesssim \nu^{-1}(\eps\nu)^{\alpha}\|a^\eps\|_{L^\infty(\tilde B^{\f32,-}_{2,\eps\nu})}
\|u^\eps\|_{L^1(\dot B^{\f52}_{2,1})}.
\end{equation}
Finally,
$$
 \|\frac{\eps a^\eps(\theta^\eps-a^\eps)}{1+\eps a^\eps}\nabla a^\eps\|_{L^1(\tilde B^{-\f12-\alpha,+}_{2,\eps\nu})}
 \lesssim \|\nabla a^\eps\|_{L^2(\dot B^{\f12}_{2,1})}\|\theta^\eps-a^\eps\|_{L^2(\tilde B^{\f12,+}_{2,\eps\nu})}\|\eps a^\eps\|_{L^\infty(\dot B^{\f32-\alpha}_{2,1})}.
 $$
 Hence using again that $\tilde B^{-\f12,+}_{2,\eps\nu}\hookrightarrow \tilde B^{s-2,+}_{p,\eps\nu}$
 and \eqref{eq:alpha}, we conclude that
 \begin{align}\label{eq:v10}
  \|\frac{\eps a^\eps(\theta^\eps\!-\!a^\eps)}{1\!+\!\eps a^\eps}\nabla a^\eps\|_{L^1(\tilde B^{s-2,+}_{p,\eps\nu})}
  \hspace{8.5cm}\\\hspace{1cm}
  \lesssim\nu^{-1}
 (\eps\nu)^{\alpha}\|a^\eps\|_{L^\infty(\tilde B^{\f32,-}_{\eps\nu})} \|a^\eps\|_{L^2(\dot B^{\f32}_{2,1})}
(\|\theta^\eps\|_{L^2(\tilde B^{\f12,+}_{2,\eps\nu})}+\|a^\eps\|_{L^2(\dot B^{\f32}_{2,1})}).\nonumber
\end{align}
So putting together inequalities \eqref{eq:v1} to \eqref{eq:v10},  we end up with
$$
\displaylines{\nu\|\dv\|_{L^1(\tilde B^{s,+}_{p,\eps\nu})}+\|\dv\|_{\tilde L^\infty(\tilde B^{s-2,+}_{p,\eps\nu})}
\lesssim  \|\dv_0\|_{\tilde B^{s-2,+}_{p,\eps\nu}}\hfill\cr\hfill+M_0(\|(\dv,\dT)\|_{L^2(\tilde B^{s-1,+}_{p,\eps\nu})}+\|\nabla\dV\|_{L^2(\dot B^{s-1}_{p,1})})+(\eps\nu)^\alpha M_0^2(1+\nu^{-1}M_0).}
$$
Bearing in mind \eqref{eq:theta1}, we thus  see that if $M_0$ is small enough with respect to $\nu,$
\begin{align}\label{eq:convergence}
\nu^{\f12}\|\dT\|_{\tilde L^2(\tilde B^{s-1,+}_{p,\eps\nu})}+\|\dT\|_{\tilde L^\infty(\tilde B^{s-2,+}_{p,\eps\nu})}+
\nu\|\dv\|_{L^1(\tilde B^{s,+}_{p,\eps\nu})}+\|\dv\|_{\tilde L^\infty(\tilde B^{s-2,+}_{p,\eps\nu})}
\lesssim \|(\dT_0,\dv_0)\|_{\tilde B^{s-2,+}_{p,\eps\nu}}\nonumber\\\qquad
+M_0^2(\eps\nu)^\alpha
+M_0\|\nabla\dV\|_{L^2(\dot B^{s-1,+}_{p,1})}
+\|\pa_t\dV\|_{L^1(\tilde B^{s-2,+}_{p,\eps\nu})+L^2(\tilde B^{s-3,+}_{p,\eps\nu})}
\end{align}
 whenever $s>1/2,$ $\,4/p-1/2\leq s\leq 3/p\,$ and $\,2\leq p<\infty.$
This completes the proof of the theorem.


\subsection{The case of smoother data}

In order to improve the results of convergence
(see Remark \ref{r:smoothbis}),
we need  to have higher order  a priori estimates for the linear system \eqref{eq:lhNS}.
In effect, if we want to have convergence in \eqref{eq:comp3} for the norm $\tilde L^{\f{2p}{p-2}}(\dot B^{\f2p-\f12}_{p,1})$
rather than  $\tilde L^{\f{2p}{p-2}}(\tilde B^{\f2p-\f32,+}_{p,\eps\nu})$
then we need $\theta$ to have the same regularity as $b,$ namely  $\dot B^{\f32}_{2,1}.$
So we need in addition that $\theta_0\in\tilde B^{\f32,-}_1$ and, owing to linear coupling,
this will enforce us to take $u_0\in \tilde B^{\f32,-}_1.$

Here we just point out what has to be modified to our previous arguments
so as to handle such data.
Let us start with \eqref{eq:bdtheta}. We concentrate on the high frequency regime.
First we notice that
$$
\pa_t\theta-\tilde\kappa\Delta\theta=-\Lambda d.
$$
Hence standard energy estimates ensure that
$$
\|\Lambda\theta_j(t)\|_{L^2}+\tilde\kappa 2^{2j}\|\Lambda\theta_j\|_{L^1_t(L^2)}
\leq \|\Lambda\theta_j(0)\|_{L^2}+\|\Lambda^2d_j\|_{L^1_t(L^2)}.
$$
Taking advantage of \eqref{eq:esthigh}, we thus get
\begin{equation}\label{eq:esthighbis}
2^j\|\theta_j(t)\|_{L^2}+2^{3j}\|\theta_j\|_{L^1_t(L^2)}\leq C\|(2^jb_j,d_j,2^j\theta_j)(0)\|_{L^2}.
\end{equation}
We also need more regularity for $(b,d).$ This is given by \eqref{eq:esthigh} after multiplying by $2^j$:
\begin{equation}\label{eq:esthighter}
\|(2^{2j}b_j,2^jd_j,\theta_j)(t)\|_{L^2}+\Int_0^t\|(2^{2j}b_j,2^{3j}d_j,2^{2j}\theta_j)\|_{L^2}\,d\tau
\leq C\|(2^{2j}b_j,2^jd_j,\theta_j)(0)\|_{L^2}.
\end{equation}
Arguing as in the proof of Proposition \ref{p:paralinear}, we thus deduce that
$$\displaylines{
\|b\|_{\tilde L^\infty_t(\tilde B^{s+1,-}_1\cap\tilde B^{s+2,-}_1)}+\|(d,w,\theta)\|_{\tilde L^\infty_t(\tilde B^{s+1,-}_1)}
+\int_0^t\bigl(\|b\|_{\tilde B^{s+1,+}_1\cap\tilde B^{s+2,+}_1}+\|(d,w,\theta)\|_{\tilde B^{s+3,-}_1}\bigr)\,d\tau
 \hfill\cr\hfill\leq Ke^{CV(t)}\biggl(\|b_0\|_{\tilde B^{s+1,-}_1\cap\tilde B^{s+2,-}_1}+\|(d_0,w_0,\theta_0)\|_{\tilde B^{s+1,-}_1}
\hfill\cr\hfill +\int_0^te^{-CV(\tau)}\bigl(\|B\|_{\tilde B^{s+1,-}_1\cap B^{s+2,-}_1}+\|(D,W,G)\|_{\tilde B^{s+1,-}_1}\bigr)\,d\tau\biggr)\cdotp}
 $$
Starting from this inequality and following the computations of Subsection \ref{ss:global},
it is easy to get the result of Remark \ref{r:smooth}. Next, resorting to the first inequality of
Proposition \ref{p:strichartz} with $s=1/2$ and to nonlinear estimates,  we get Remark \ref{r:smoothbis}.


\section{The nonconducting case}\label{s:zero}

As pointed out in the introduction, in the case $\kappa=0,$ it is easier to work with~$\cR^\eps.$
The reason why is that  the linearized equations for  $(u^\eps, \cR^\eps)$  are the same as those of the classical
 barotropic Navier-Stokes equations (see next paragraph).
 Apart from this purely technical point and the fact that one has to work with smoother data,
 the overall approach for investigating the  global existence and low Mach number
issues is the same : first we perform the change of variables
\begin{equation}\label{eq:changebis}
(a,u,\cR)(t,x)=\eps(a^\eps,u^\eps,\cR^\eps)(\eps^2\nu t,\eps\nu x)
\quad\hbox{and}\quad \widetilde{V}(t,x)=\eps V^\eps(\eps^2\nu t,\eps\nu x),
\end{equation}
so as to reduce the proof of existence to  the case $\eps=\nu=1,$
and next we take advantage of dispersive properties of the acoustic wave equation, and of
parabolic estimates to establish the convergence to some suitable solution
of the Boussinesq system with no heat conduction (namely \eqref{eq:nBouS}).


\subsection{Linear and paralinear estimates}

 If we  decompose,  as in the heat-conducting case, the velocity field $u$ into its (reduced) potential part $d,$ and its divergence-free part $w,$
then the linearized system about $0$ reads
 \begin{equation}\label{eq:linear0}
  \left\{
    \begin{aligned}
      & \pa_t  a +\Lambda d =0, \\
      &\pa_t d - \Delta d-\Lambda \mathcal{R} =0,\\
      &\pa_t \mathcal{R}  +\Lambda d =0,\\
      &\pa_tw-\tilde\mu\Delta w=0.\end{aligned}\right.
\end{equation}
As in the heat-conducting case, $w$ just fulfills the heat equation.
Next, we notice that $(\cR, d)$ satisfies the linearized equation for the compressible modes of the barotropic
Navier-Stokes equations. Hence, following the method of \cite{D1}, we gather that
for some universal constant $C,$
$$
\begin{array}{lll}
\|(\cR_j,d_j)(t)\|_{L^2}+2^{2j}\Int_0^t \|(\cR_j,d_j)\|_{L^2}\,d\tau\leq C\|(\cR_j,d_j)(0)\|_{L^2}
&\hbox{if}& j\leq 0,\\[1ex]
\|(2^j\cR_j,d_j)(t)\|_{L^2}+\Int_0^t \|(2^j\cR_j,2^{2j}d_j)\|_{L^2}\,d\tau\leq C\|(2^j\cR_j,d_j)(0)\|_{L^2}
&\hbox{if}& j>0.\end{array}
$$
Now, from the first and last equations of \eqref{eq:linear0}, we see that
$$
a_j(t)-\cR_j(t)=a_j(0)-\cR_j(0)\quad\hbox{for all }\ t\in\R^+.
$$
Hence, taking advantage of the above estimate for $\cR_j,$ we get
$$
\max(1,2^j)\|a_j(t)\|_{L^2}\leq C\bigl(\max(1,2^j)\|(a_j(0),\cR_j(0))\|_{L^2}+\|d_{j}(0)\|_{L^2}\bigr).
$$
{}From   those inequalities,  arguing as in the case $\kappa>0,$
one may deduce  a priori estimates for the following paralinearized equations:
 \begin{equation}
  \left\{
    \begin{aligned}
      & \pa_t a+\Lambda d+T_{v^k}\pa_k a=A, \\
      &\pa_t d + T_{v^k}\pa_k d- \Delta d-\Lambda\cR  =D,\\
      &\pa_t \cR  +\Lambda d +T_{v^k}\pa_k \cR=R,\\
      &\pa_t w + T_{v^k}\pa_k w-\tilde\mu\Delta w=W,
    \end{aligned}
  \right.
  \label{eq:lhNSbis0}
  \end{equation}
where the source terms $A,$ $D,$ $R,$ $W$ and the vector field $v$ are given.
\medbreak
More precisely, we have
\begin{prop}\label{p:paralinear0}
Let  $\cV(t):=\int_0^t \|\na v\|_{L^\infty}\,d\tau$. There exists a constant $K$ depending only on $\tilde\mu$
and a universal constant $C$  such that for all $s\in\R,$ the following inequality holds true:
$$\displaylines{
\|(a,\cR)\|_{\tilde L^\infty_t(\tilde B^{s+1,-}_1)}+\|(d,w)\|_{\tilde L^\infty_t(\dot B^s_{2,1})}
+\int_0^t\bigl(\|(d,w)\|_{\dot B^{s+2}_{2,1}}+\|\cR\|_{\tilde B^{s+1,+}_1}\bigr)\,d\tau
 \hfill\cr\hfill\leq Ke^{C\cV(t)}\biggl(\|(a_0,\cR_0)\|_{\tilde B^{s+1,-}_1}+\|(d_0,w_0)\|_{\dot B^s_{2,1}}
\hfill\cr\hfill +\int_0^te^{-C\cV(\tau)}\bigl(\|(A,R)\|_{\tilde B^{s+1,-}_1}+\|(D,W)\|_{\dot B^s_{2,1}}\bigr)\,d\tau\biggr)\cdotp}
 $$
\end{prop}


\subsection{The proof of global existence}

Here, in the case $\eps=\nu=1,$
 we want to prove the existence of a global solution $(a,u,\cR)$ to \eqref{eq:nhNS} with
$$
a\in \wt C(\dot B^{\f12}_{2,1}\cap \dot B^{\f72}_{2,1}),\quad\!\!
u\in \wt C(\dot B^{\f12}_{2,1}\cap \dot B^{\f52}_{2,1})
\cap L^1(\dot B^{\f52}_{2,1}\cap \dot B^{\f92}_{2,1}),\quad\!\!
\cR\in\wt C(\dot B^{\f12}_{2,1}\cap \dot B^{\f72}_{2,1})
\cap L^1(\tilde B^{\f32,+}_{1}\cap \tilde B^{\f72,+}_{1}).
$$
For that, this is mainly a matter of proving a priori estimates in this space, taking for granted
the existence of a solution.
Indeed,  the a priori estimates that we are going to prove below would be the same for the system
truncated by means of the Friedrichs method (see e.g. \cite{BCD}, Chap. 10 for the related case of
the barotropic Navier-Stokes equation).

More precisely, we have to bound:
\begin{equation}\label{eq:X0}
X:=\|(a,\cR)\|_{\tilde L^\infty(\dot B^{\f12}_{2,1}\cap \dot B^{\f72}_{2,1})}+
\|u\|_{\tilde L^\infty(\dot B^{\f12}_{2,1}\cap \dot B^{\f52}_{2,1})}+\|u\|_{L^1(\dot B^{\f52}_{2,1}\cap \dot B^{\f92}_{2,1})}
+\|\cR\|_{L^1(\tilde B^{\f32,+}_{1}\cap \tilde B^{\f72,+}_{1})}.
\end{equation}

As we have in mind to apply Proposition \ref{p:paralinear0} (twice: once with $s=3/2$ and once with $s=7/2$),
we rewrite  \eqref{eq:nhNS} as follows:
 \begin{equation}  \label{eq:lnhNS}
  \left\{
    \begin{aligned}
      & \pa_t  a +T_{u^k}\pa_k a+\Lambda d =A, \\
      &\pa_t d + T_{u^k}\pa_k d- \Delta d-\Lambda \mathcal{R} =D,\\
      &\pa_t \mathcal{R}  +\Lambda d +T_{u^k}\pa_k  \mathcal{R} =R,\\
      &\pa_t w+T_{u^k}\pa_k w-\mu\Delta w= W,
    \end{aligned}
  \right.
  \end{equation}
  where \beno A&:=&T_{u^k}\pa_k a -u\cdot\na a-a\div u, \\
  D&:=& T_{u_k}\pa_k d-\Lambda^{-1}\div(u\cdot\na u)
   -\Lambda^{-1}\div\bigg[\f{a}{1+a}\bigl(\tilde\mu\Delta u+(\tilde\lambda+\tilde\mu)\nabla\div u\bigr)- \f{a\na (\mathcal{R}+\widetilde{V}) }{(1+a)}\bigg], \\ R&:=&T_{u^k}\pa_k \mathcal{R}-u\cdot\na \mathcal{R}-\mathcal{R}\div u-\pa_t \widetilde{V}-\div(\widetilde{V}u) +  [2\tilde\mu|Du|^2+\tilde\lambda (\div u)^2],\\
 W&:=&T_{u^k}\pa_k w-\cP(u\cdot\na u)-\cP\bigg[\f{a}{1+ a}\bigl(\tilde\mu\Delta u+(\tilde\lambda+\tilde\mu)\nabla\div u\bigr)-\f{a\na(\mathcal{R}+\widetilde{V})}{(1+a)} \bigg]\cdotp\eeno

According to Proposition \ref{p:paralinear0}, we thus have to
 bound $A,R$ in $L^1(\dot B^{\f12}_{2,1}\cap \dot B^{\f72}_{2,1})$
and  $D,W$ in $L^1(\dot B^{\f12}_{2,1}\cap \dot B^{\f52}_{2,1}).$
We shall assume throughout that $\|a\|_{L^\infty(\R^+\times\R^3)}$ is small.

\subsubsection*{Bounds for $\|A\|_{L^1(\dot B^{\f12}_{2,1}\cap\dot B^{\f72}_{2,1})}$}

Recall that
$$
A=-T'_{\pa_ka}u^k-a\div u.
$$
Using standard product laws for the paraproduct and remainder (see e.g. \cite{BCD}), we get
\begin{align}\label{eq:G1a}
\|T'_{\pa_ka}u^k\|_{L^1(\dot B^{\f12}_{2,1})}&\lesssim\|\nabla a\|_{L^\infty(\dot B^{-\f12}_{2,1})}\|u\|_{L^1(\dot B^{\f52}_{2,1})},\\\label{eq:G1b}
\|T'_{\pa_ka}u^k\|_{L^1(\dot B^{\f72}_{2,1})}&\lesssim\|\nabla a\|_{L^\infty(\dot B^{\f12}_{2,1})}\|u\|_{L^1(\dot B^{\f92}_{2,1})},\\\label{eq:G1c}
\|a\div u\|_{L^1(\dot B^{\f12}_{2,1})}&\lesssim\|a\|_{L^\infty(\dot B^{\f12}_{2,1})}\|\div u\|_{L^1(\dot B^{\f32}_{2,1})},\\\label{eq:G1d}
\|a\div u\|_{L^1(\dot B^{\f72}_{2,1})}&\lesssim\|a\|_{L^\infty(\dot B^{\f32}_{2,1})}\|\div u\|_{L^1(\dot B^{\f72}_{2,1})}
+\|a\|_{L^\infty(\dot B^{\f72}_{2,1})}\|\div u\|_{L^1(\dot B^{\f32}_{2,1})}.
\end{align}
Hence
\begin{equation}\label{eq:G1}
\|A\|_{L^1(\dot B^{\f12}_{2,1}\cap\dot B^{\f72}_{2,1})}\lesssim \|a\|_{L^\infty(\dot B^{\f12}_{2,1}\cap \dot B^{\f72}_{2,1})}
\|u\|_{L^1(\dot B^{\f52}_{2,1}\cap \dot B^{\f92}_{2,1})}.
\end{equation}

\subsubsection*{Bounds for $\|D\|_{L^1(\dot B^{\f12}_{2,1}\cap\dot B^{\f52}_{2,1})}$
and $\|W\|_{L^1(\dot B^{\f12}_{2,1}\cap\dot B^{\f52}_{2,1})}$}

We may rewrite $D$ as follows:
$$
 D= [T_{u_k},\Lambda^{-1}\pa_i]\pa_k u^i-\Lambda^{-1}\pa_iT'_{\pa_ku^i}u^k
   -\Lambda^{-1}\div\bigg[\f{a}{1+a}\bigl(\tilde\mu\Delta u+(\tilde\lambda+\tilde\mu)\nabla\div u- \na (\mathcal{R}+V)\bigr)\bigg]\cdotp
   $$
The first two terms of $D$ may be treated as in \eqref{eq:p4}: we get for any $s>0,$
\begin{equation}\label{eq:G2a}
\| [T_{u_k},\Lambda^{-1}\pa_i]\pa_k u^i-\Lambda^{-1}\pa_iT'_{\pa_ku^i}u^k\|_{\dot B^{s}_{2,1}}\lesssim\|\nabla u\|_{L^\infty}\|u\|_{\dot B^s_{2,1}}.
\end{equation}
Next, classical composition and  tame estimates yield for $s>0,$
$$
\| \frac a{1+a}\cA u\|_{\dot B^s_{2,1}}\lesssim \|a\|_{L^\infty}\|\cA u\|_{\dot B^s_{2,1}}
+\|\cA u\|_{L^\infty}\|a\|_{\dot B^s_{2,1}}.
$$
Hence, using the embedding $\dot B^{\f32}_{2,1}\hookrightarrow L^\infty,$ we easily get
\begin{equation}\label{eq:G2b}
\| \frac a{1+a}\cA u\|_{L^1(\dot B^{\f12}_{2,1}\cap \dot B^{\f52}_{2,1})}
\lesssim \|a\|_{L^\infty(\dot B^{\f12}_{2,1}\cap \dot B^{\f52}_{2,1})}
 \|u\|_{L^1(\dot B^{\f52}_{2,1}\cap \dot B^{\f92}_{2,1})}.
\end{equation}
Finally, we have
\begin{eqnarray}\label{eq:G2c}
&&\| \frac a{1+a} \na (\mathcal{R}+\widetilde{V})\|_{L^1(\dot B^{\f12}_{2,1})}\lesssim\|a\|_{L^\infty(\dot B^{\f12}_{2,1})}\|\nabla(\cR+\widetilde{V})\|_{L^1(\dot B^{\f32}_{2,1})},\\
\label{eq:G2d}&&\| \frac a{1+a} \na (\mathcal{R}+\widetilde{V})\|_{L^1(\dot B^{\f52}_{2,1})}\lesssim
\|a\|_{L^\infty(\dot B^{\f32}_{2,1})}\|\nabla(\cR+\widetilde{V})\|_{L^1(\dot B^{\f52}_{2,1})}\\&&\hspace{7cm}
+\|a\|_{L^\infty(\dot B^{\f52}_{2,1})}\|\nabla(\cR+\widetilde{V})\|_{L^1(\dot B^{\f32}_{2,1})}.\nonumber
\end{eqnarray}
So putting \eqref{eq:G2a} to \eqref{eq:G2d} together, we get
\begin{equation}\label{eq:G2}
\begin{aligned}&\|D\|_{L^1(\dot B^{\f12}_{2,1}\cap \dot B^{\f52}_{2,1})}\lesssim
\|u\|_{L^1(\dot B^{\f52}_{2,1})}\|u\|_{L^\infty(\dot B^{\f12}_{2,1}\cap \dot B^{\f52}_{2,1})}
\\&\hspace{1cm}+\|a\|_{L^\infty(\dot B^{\f12}_{2,1}\cap\dot B^{\f52}_{2,1})}
\bigl(\|u\|_{L^1(\dot B^{\f52}_{2,1}\cap\dot B^{\f92}_{2,1})}+\|\cR\|_{L^1(\dot B^{\f52}_{2,1}\cap \dot B^{\f72}_{2,1})}
+\|\nabla \widetilde{V}\|_{L^1(\dot B^{\f32}_{2,1}\cap \dot B^{\f52}_{2,1})}\bigr).
\end{aligned}\end{equation}
It is clear that $W$ satisfies exactly the same inequality.

\subsubsection*{Bounds for $\|R\|_{L^1(\dot B^{\f12}_{2,1}\cap\dot B^{\f72}_{2,1})}$}

Recall that
$$ R=-T'_{\pa_k\cR}u^k-\mathcal{R}\div u-\pa_t \widetilde{V}-\div(\widetilde{V}u) +  [2\tilde\mu|Du|^2+\tilde\lambda (\div u)^2].$$
First we have for any $s>0,$
$$
\|T'_{\pa_k\cR}u^k\|_{\dot B^s_{2,1}}\lesssim \|\nabla\cR\|_{L^\infty}\|u\|_{\dot B^s_{2,1}}.
$$
Hence
\begin{equation}\label{eq:G3a}
\begin{aligned}
&\|T'_{\pa_k\cR}u^k\|_{L^1(\dot B^{\f12}_{2,1})}\lesssim \|\cR\|_{L^1(\dot B^{\f52}_{2,1})}\|u\|_{L^\infty(\dot B^{\f12}_{2,1})},\\
&\|T'_{\pa_k\cR}u^k\|_{L^1(\dot B^{\f72}_{2,1})}\lesssim \|\cR\|_{L^\infty(\dot B^{\f52}_{2,1})}\|u\|_{L^1(\dot B^{\f72}_{2,1})}.
\end{aligned}
\end{equation}
Next, product estimates imply that
\begin{equation}\label{eq:G3b}\begin{aligned}
&\|\cR\div u\|_{L^1(\dot B^{\f12}_{2,1})}\lesssim \|\cR\|_{L^2(\dot B^{\f32}_{2,1})}\|u\|_{L^2(\dot B^{\f32}_{2,1})},\\
&\|\cR\div u\|_{L^1(\dot B^{\f72}_{2,1})}\lesssim \|\cR\|_{L^\infty(\dot B^{\f32}_{2,1})}\|u\|_{L^1(\dot B^{\f92}_{2,1})}
+\|u\|_{L^1(\dot B^{\f52}_{2,1})}\|\cR\|_{L^\infty(\dot B^{\f72}_{2,1})}.
\end{aligned}
\end{equation}
We also have
\begin{align}\label{eq:G3c}
\|\div(\widetilde{V}u)\|_{L^1(\dot B^{\f12}_{2,1})}&\lesssim\|\widetilde{V}\|_{L^\infty(\dot B^{\f12}_{2,1})}\|u\|_{L^1(\dot B^{\f52}_{2,1})}
+\|\widetilde{V}\|_{L^1(\dot B^{\f52}_{2,1})}\|u\|_{L^\infty(\dot B^{\f12}_{2,1})}\\\label{eq:G3d}
\|\div(\widetilde{V}u)\|_{L^1(\dot B^{\f72}_{2,1})}&\lesssim\|\widetilde{V}\|_{L^\infty(\dot B^{\f32}_{2,1})}\|u\|_{L^1(\dot B^{\f92}_{2,1})}
+\|\widetilde{V}\|_{L^1(\dot B^{\f92}_{2,1})}\|u\|_{L^\infty(\dot B^{\f32}_{2,1})}.
\end{align}
And finally,
\begin{align}\label{eq:G3e}
\|\nabla u\otimes\nabla u\|_{L^1(\dot B^{\f12}_{2,1})}&\lesssim\|\nabla u\|_{L^1(L^\infty)}\|\nabla u\|_{L^\infty(\dot B^{\f12}_{2,1})},\\
\label{eq:G3f}
\|\nabla u\otimes\nabla u\|_{L^1(\dot B^{\f72}_{2,1})}&\lesssim\|\nabla u\|_{L^\infty(L^\infty)}\|\nabla u\|_{L^1(\dot B^{\f72}_{2,1})}.
\end{align}
Therefore, combining inequalities \eqref{eq:G3a} to \eqref{eq:G3f}, and using embedding, we end up with
\begin{equation}\label{eq:G3}
\begin{aligned}
\|R\|_{L^1(\dot B^{\f12}_{2,1}\cap\dot B^{\f72}_{2,1})}\lesssim
\|u\|_{L^1(\dot B^{\f52}_{2,1}\cap \dot B^{\f92}_{2,1})}\|u\|_{L^\infty(\dot B^{\f12}_{2,1}\cap \dot B^{\f52}_{2,1})}
+\|\pa_t\widetilde{V}\|_{L^1(\dot B^{\f12}_{2,1}\cap\dot B^{\f72}_{2,1})}\hspace{1.2cm}\\
+\|\widetilde{V}\|_{L^1(\dot B^{\f52}_{2,1}\cap\dot B^{\f92}_{2,1})}\|u\|_{L^\infty(\dot B^{\f12}_{2,1}\cap \dot B^{\f32}_{2,1})}
+\|u\|_{L^1(\dot B^{\f52}_{2,1}\cap\dot B^{\f92}_{2,1})}\|\widetilde{V}\|_{L^\infty(\dot B^{\f12}_{2,1}\cap \dot B^{\f32}_{2,1})}\\
\!\!\!\!\!+\|\cR\|_{L^\infty(\dot B^{\f12}_{2,1}\cap\dot B^{\f72}_{2,1})}\|u\|_{L^1(\dot B^{\f52}_{2,1}\cap\dot B^{\f92}_{2,1})}
\!\!+\! \|\cR\|_{L^2(\dot B^{\f32}_{2,1})}\|u\|_{L^2(\dot B^{\f32}_{2,1})}
\!\!+\! \|\cR\|_{L^1(\dot B^{\f52}_{2,1})}\|u\|_{L^\infty(\dot B^{\f12}_{2,1})}.\!\!\!\!\!\!\!
\end{aligned}
\end{equation}

Putting \eqref{eq:G1}, \eqref{eq:G2} and \eqref{eq:G3} together, one may finally conclude that
for some constant $K$ depending only on $\tilde\lambda$ and $\tilde\mu,$ we have
$$X\leq K\Bigl(X(0)+X^2+\bigl(\|\widetilde{V}\|_{L^\infty(\dot B^{\f12}_{2,1}\cap \dot B^{\f32}_{2,1})}
+\|\widetilde{V}\|_{L^1(\dot B^{\f52}_{2,1}\cap\dot B^{\f92}_{2,1})}\bigr)X
+\|\pa_t\widetilde{V}\|_{L^1(\dot B^{\f12}_{2,1}\cap\dot B^{\f72}_{2,1})}\Bigr).
$$
{}From this, we see that if $X(0)$ and the terms pertaining to $\widetilde{V}$ are small enough, then
\begin{equation}\label{eq:Xest}
X\leq 2K\Bigl(X(0)+\|\pa_t\widetilde{V}\|_{L^1(\dot B^{\f12}_{2,1}\cap\dot B^{\f72}_{2,1})}\Bigr)\cdotp
\end{equation}
Going back to the original variables according to \eqref{eq:changebis}, we then
get the global existence part of Theorem \ref{th:main3}, for any $\eps>0.$


\subsection{The proof of convergence}

As in the case where $\kappa>0,$ we first show that $(\cQ u^\eps,\cR^\eps)$
goes to $0,$ a consequence of Strichartz estimates, then establish that $(\cP u^\eps,\Theta^\eps)$
goes to the solution $(v,\Theta)$ of the Boussinesq system \eqref{eq:nBouS}.

\subsubsection{Convergence to $0$ for $(\cQ u^\eps,\cR^\eps)$}

It suffices to prove  dispersion estimates in the case $\eps=1.$ The change
of variable \eqref{eq:changebis} will provide us with decay estimates in the general case.
Now, the system  for $(\cQ u,\cR)$ reads
$$\left\{
\begin{aligned}&\pa_t\cQ u+\nabla\cR=-\cQ(u\cdot\nabla u)-\cQ\biggl(\frac{\cA u}{1+a}\biggr)+\cQ\biggl(\frac a{1+a}\nabla(f+\cR)\biggr)=:\bH_1,\\
&\pa_t\cR+\div\cQ u=-\pa_tV-\div((V+\cR)u)-2\mu|Du|^2-(\lambda+\mu)(\div u)^2=:\bH_2.
\end{aligned}\right.
$$
Therefore, Strichartz estimates imply that for all $p\in[2,\infty),$
\begin{equation}\label{eq:stric1}
\|(\cQ u,\cR)\|_{\wt L^{\frac{2p}{p-2}}(\dot B^{\frac2p-\frac12}_{p,1})}\lesssim
\|(\cQ u_0,\cR_0)\|_{\dot B^{\f12}_{2,1}}+\|(\bH_1,\bH_2)\|_{L^1(\dot B^{\f12}_{2,1})}.
\end{equation}
So it is only a matter of bounding $\bH_1$ and $\bH_2$ in $L^1(\dot B^{\f12}_{2,1}),$
which may be done by using standard results of continuity in Besov
spaces and the fact that $\cQ$ is an homogeneous multiplier of degree $0.$
More precisely, we have
$$
\begin{aligned}
\|\cQ(u\cdot\nabla u)\|_{L^1(\dot B^{\f12}_{2,1})}&\lesssim\|u\|_{L^\infty(\dot B^{\f12}_{2,1})}
\|u\|_{L^1(\dot B^{\f52}_{2,1})},\\
\| \cQ\biggl(\frac{\cA u}{1+a}\biggr)   \|_{L^1(\dot B^{\f12}_{2,1})}&\lesssim
(1+\|a\|_{L^\infty(\dot B^{\f32}_{2,1})})\|u\|_{L^1(\dot B^{\f52}_{2,1})},\\
\|\cQ\Bigl(\frac a{1+a}\nabla(V+\cR)\Bigr)\|_{L^1(\dot B^{\f12}_{2,1})}&\lesssim
\|a\|_{L^\infty(\dot B^{\f12}_{2,1})}(\|\nabla V\|_{L^1(\dot B^{\f32}_{2,1})}+\|\nabla\cR\|_{L^1(\dot B^{\f32}_{2,1})}),\\
\|\div((V+\cR)u)\|_{L^1(\dot B^{\f12}_{2,1})}&\lesssim\|u\|_{L^1(\dot B^{\f52}_{2,1})}\|V+\cR\|_{L^\infty(\dot B^{\f12}_{2,1})}
+\|u\|_{L^\infty(\dot B^{\f12}_{2,1})}\|V+\cR\|_{L^1(\dot B^{\f52}_{2,1})},\\
\|\nabla u\otimes \nabla u\|_{L^1(\dot B^{\f12}_{2,1})}&\lesssim\|u\|_{L^1(\dot B^{\f52}_{2,1})}\|u\|_{L^\infty(\dot B^{\f32}_{2,1})}.
\end{aligned}
$$
Therefore, if we set
$$
C_0=\!\|(a_0,\cR_0)\|_{\dot B^{\f12}_{2,1}\cap\dot B^{\f72}_{2,1}}\!+\|u_0\|_{B^\f12_{2,1}\cap\dot B^\f52_{2,1}}
\!+ \|\pa_tV \|_{L^1(\dot B^\f12_{2,1}\cap\dot B^\f72_{2,1})},
$$
then plugging the above inequalities in \eqref{eq:stric1} and using \eqref{eq:Xest} leads to
$$
\|(\cQ u,\cR)\|_{\tilde L^{\frac{2p}{p-2}}(\dot B^{\frac2p-\frac12}_{p,1})}\leq KC_0\quad\hbox{for all }\ p\in[2,\infty).
$$
{}From \eqref{eq:Xest}, we also know that $(\cQ u,\cR)$ is  bounded by $KC_0$  in $L^1(\dot B^{\f52}_{2,1}).$ Hence using interpolation exactly
as in the case $\kappa>0$ leads to
$$
\|(\cQ u,\cR)\|_{\tilde L^{2}(\dot B^{s}_{p,1})}\leq KC_0\quad\hbox{for all }\ p\geq2\quad\hbox{and}\quad
s\in[-1/2+4/p,3/p].
$$
Now, going back to the original variables, we gather that for $\eps>0,$ we have
\begin{align}\label{eq:stric0}
\|(\cQ u^\eps,\cR^\eps)\|_{\tilde L^{\frac{2p}{p-2}}(\dot B^{\frac2p-\frac12}_{p,1})}&\leq KC_0^\eps\eps^{\frac12-\frac1p}\quad\hbox{if }\ 2\leq p<\infty,\\
\nu^{\f12}\|(\cQ u^\eps,\cR^\eps)\|_{\tilde L^{2}(\dot B^{s}_{p,1})}&\leq
KC_0^\eps(\eps\nu)^{\f3p-s}\quad\hbox{for all }\ p\geq2\ \hbox{ and }\
s\in[-1/2+4/p,3/p].
\end{align}
with $C_0^\eps$ defined in the statement of Theorem \ref{th:main3}.


\subsubsection{Global existence of a solution to  \eqref{eq:nBouS}}

Under the assumption that \eqref{eq:smallnBouS}, the existence of a global solution $(\Theta,v)$ to
\eqref{eq:nBouS} satisfying \eqref{eq:nBouS1} is an easy modification of the
corresponding proof for the standard incompressible Navier-Stokes equations, combined with
the following a priori estimate for the transport equation (see e.g. \cite{BCD}, Chap. 3):
$$
\|\Theta\|_{\tilde L^\infty_T(\dot B^{\f12}_{2,1})}\leq \|\Theta_0\|_{\dot B^{\f12}_{2,1}}\exp\biggl(\int_0^T\|v\|_{\dot B^{\f52}_{2,1}}\,dt\biggr)\cdotp
$$
Indeed, using once again Proposition \ref{p:heatestimates} and product estimates, we see that
$$
\displaylines{\|u\|_{\tilde L^\infty_T(\dot B^{\f12}_{2,1})}+\mu \|u\|_{L^1_T(\dot B^{\f52}_{2,1})}
\lesssim \|u_0\|_{\dot B^{\f12}_{2,1}}+ \|u\|_{L^\infty_T(\dot B^{\f12}_{2,1})}
\|u\|_{L^1_T(\dot B^{\f52}_{2,1})}
+\|\Theta\|_{L^\infty_T(\dot B^{\f12}_{2,1})} \|\nabla V\|_{L^1_T(\dot B^{\f32}_{2,1})}.}
$$
Hence if \eqref{eq:smallnBouS} is fulfilled then  one may close the a priori estimates globally in time.


\subsubsection{Convergence of  $(\cP u^\eps,\Theta^\eps)$}

Let us first notice that (recall that $\Theta^\eps=a^\eps-\cR^\eps-V^\eps$)
$$\begin{array}{lll}
\cP\biggl(\frac{a^\eps}{1+\eps a^\eps}\nabla(V^\eps\!+\!\cR^\eps)\biggr)&=&\cP\biggl( a^\eps\nabla (V^\eps\!+\!\cR^\eps)\biggr)-\cP\biggl(a^\eps\frac{\eps a^\eps}{1+\eps a^\eps}\nabla (V^\eps\!+\!\cR^\eps)\biggr)\\
&=&\cP\bigl(\Theta^\eps\nabla(V^\eps+\cR^\eps)\bigr)-\cP\biggl(a^\eps\frac{\eps a^\eps}{1+\eps a^\eps}\nabla(V^\eps\!+\!\cR^\eps)\biggr).
\end{array}
$$
Therefore the system for $(\dT,\dv):=(\Theta^\eps-\Theta,\cP u^\eps-v)$ writes
$$\left\{
\begin{array}{l}
\pa_t\dT+\cP u^\eps\cdot\nabla\dT=-\dv\cdot\nabla\Theta-\cQ u^\eps\cdot\nabla\Theta^\eps-\Theta^\eps\div\cQ u^\eps
-\eps\bigl(2\mu|Du^\eps|^2+\lambda(\div u^\eps)^2\bigr),\\[2ex]
\pa_t\dv-\mu\Delta\dv+\cP(v\cdot\nabla\dv)+\cP(\dv\cdot\nabla\cP u^\eps)
=\cP(\dT\nabla V+\Theta^\eps\nabla\dV+\Theta^\eps\nabla\cR^\eps)\\
\hspace{3.8cm}-\cP\biggl(\cQ u^\eps\cdot\nabla\cP u^\eps+ u^\eps\cdot\nabla\cQ u^\eps+\frac{\eps a^\eps}{1+\eps a^{\eps}}\bigl(\cA u^\eps
+a^\eps\nabla(V^\eps\!+\!\cR^\eps) \bigr)\biggr)\cdotp
\end{array}\right.
$$
In contrast with the heat-conducting case, we do not know how to prove convergence
\emph{globally in time}. This is due to the fact that  some terms
in the right-hand side of the equations for $(\dT,\dv)$ decay to $0$ \emph{only in $L^2$-in time spaces}
and that $\dT$ satisfies a mere transport equation (hence the r.h.s. should be bounded in
\emph{$L^1$-in-time space} if we want to get  a time independent bound for $\dT$).

We claim nevertheless that $\Theta^\eps\rightarrow\Theta$ in $\tilde L^\infty_{loc}(\dot B^{s-2}_{p,1})$ with $s$ as in the
previous step, and that $\cP u^\eps\rightarrow v$ in
$$\bigl(\tilde L^\infty_{loc}(\dot B^{s-1}_{p,1})\cap \tilde L^2_{loc}(\dot B^s_{p,1})\bigr)+
\bigl(\tilde L^\infty_{loc}(\dot B^{s-2}_{p,1})\cap L^1_{loc}(\dot B^s_{p,1})\bigr).
$$
Let us first examine $\dT.$
Denoting by ${\mathbb K}_1$ the r.h.s. of the equation for $\dT,$
standard estimates for the transport equation ensure that, if $s>-1/2$
then we have for all $T\geq0,$
\begin{equation}\label{eq:dT0}
\|\dT\|_{\tilde L^\infty_T(\dot B^{s-2}_{p,1})}
\leq  \exp\biggl(\int_0^T\|\nabla\cP u^\eps\|_{\dot B^{\f32}_{2,1}}\,dt\biggr)\biggl(\|\dT_0\|_{\dot B^{s-2}_{p,1}}
+\int_0^T\|{\mathbb K}_1\|_{\dot B^{s-2}_{p,1}}\,dt\biggr)\cdotp
\end{equation}
Product laws give if, in addition, $s>1/2,$
$$
\begin{array}{lll}
\|\dv\cdot\nabla\Theta\|_{\dot B^{s-2}_{p,1}}&\lesssim&\|\dv\|_{\dot B^s_{p,1}}\|\nabla\Theta\|_{\dot B^{-\f12}_{2,1}},\\[1.5ex]
\|\cQ u^\eps\cdot\nabla\Theta^\eps \|_{\dot B^{s-2}_{p,1}}&\lesssim&\|\cQ u^\eps\|_{\dot B^s_{p,1}}\|\nabla\Theta^\eps\|_{\dot B^{-\f12}_{2,1}},\\[1.5ex]
\|\Theta^\eps\div\cQ u^\eps \|_{\dot B^{s-2}_{p,1}}&\lesssim&\|\div\cQ u^\eps\|_{\dot B^{s-1}_{p,1}}\|\Theta^\eps\|_{\dot B^{\f12}_{2,1}}.
\end{array}
$$
For the last term of ${\mathbb K}_1,$ we use the fact that
the product maps $\dot B^{-\f12}_{2,1}\times \dot B^{\f32-\alpha}_{2,1}$ in $\dot B^{-\f12-\alpha}_{2,1}$ if $0\leq\alpha<1.$
Hence using the embedding $\dot B^{-\f12-\alpha}_{2,1}\hookrightarrow\dot B^{s-2}_{p,1}$ with $\alpha=3/p-s,$ we get
$$
\|2\mu|Du^\eps|^2+\lambda(\div u^\eps)^2\|_{\dot B^{s-2}_{p,1}}\lesssim \|u^\eps\|_{\dot B^{\f12}_{2,1}}\|u^\eps\|_{\dot B^{\f52-\alpha}_{2,1}}.
$$
Inserting those inequalities in \eqref{eq:dT0} and keeping in mind that $\nabla\cP u^\eps$ is uniformly
bounded in $L^1(\dot B^{\f32}_{2,1}),$ we get for any $s\in[-\f12+\f4p,\f3p]\cap(\f12,\infty)$:
$$
\displaylines{\quad\|\dT\|_{\tilde L^\infty_T(\dot B^{s-2}_{p,1})}
\lesssim \|\dT_0\|_{\dot B^{s-2}_{p,1}}+\int_0^T\|\dv\|_{\dot B^s_{p,1}}\|\Theta\|_{\dot B^{\f12}_{2,1}}\,dt
\hfill\cr\hfill+\int_0^T\Bigl(\|\cQ u^\eps\|_{\dot B^s_{p,1}}\|\Theta^\eps\|_{\dot B^{\f12}_{2,1}}
+\eps \|u^\eps\|_{\dot B^{\f12}_{2,1}}\|u^\eps\|_{\dot B^{\f52-\alpha}_{2,1}}\Bigr)dt,}
$$
whence
\begin{align}\label{eq:dT}
&\|\dT\|_{\tilde L^\infty_T(\dot B^{s-2}_{p,1})}
\lesssim \|\dT_0\|_{\dot B^{s-2}_{p,1}}
\\&\qquad+(1+T^{\f12})\|\Theta_0\|_{\dot B^{\f12}_{2,1}}\|\dv\|_{\tilde L^2_T(\dot B^s_{p,1})+L^1_T(\dot B^s_{p,1})}
+\nu^{-1}(C_0^\eps)^2((\nu T)^{\f12}(\eps\nu)^\alpha+\eps\nu (\nu T)^{\f\alpha2}).\!\!\!\!\nonumber
\end{align}
In order to bound $\dv,$ we shall make use once again of the parabolic estimates given by
Proposition \ref{p:heatestimates}.  The main difficulty here is that some
terms of the r.h.s.  ${\mathbb K}_2$  of the equation for $\dv$ cannot be bounded in global $L^1$-in-time spaces.
Hence we shall use the following inequality
which may be easily deduced from Proposition \ref{p:heatestimates} (we do not
track the dependency with respect to $\mu$):
$$
\|\dv\|_{\tilde L^\infty_T(\dot B^{s-1}_{p,1}+\dot B^{s-2}_{p,1})}
+\|\dv\|_{\tilde L^2_T(\dot B^s_{p,1})+L^1_T(\dot B^s_{p,1})}\lesssim
\|\dv_0\|_{\dot B^{s-1}_{p,1}+\dot B^{s-2}_{p,1}}+\|{\mathbb K}_2\|_{\wt L^2_T(\dot B^{s-2}_{p,1})+L_T^1(\dot B^{s-2}_{p,1}+\dot B^{s-1}_{p,1})}.
$$
Now, from product estimates in Besov spaces, we get
$$\begin{array}{lll}
\|v\cdot\nabla\dv\|_{L^1_T(\dot B^{s-1}_{p,1}+\dot B^{s-2}_{p,1})}&\lesssim&
(\|v\|_{ L^2_T(\dot B^{\f32}_{2,1})}+\|v\|_{L^\infty_T(\dot B^{\f12}_{2,1})})\|\dv\|_{  L^2_T(\dot B^s_{p,1})+L^1_T(\dot B^{s}_{p,1})},\\[2ex]
\|\dv\cdot\nabla\cP u^\eps\|_{L^1_T(\dot B^{s-1}_{p,1}+\dot B^{s-2}_{p,1})}&\lesssim&
\|\nabla\cP u^\eps\|_{L^1_T(\dot B^{\f32}_{2,1})}\|\dv\|_{L^\infty_T(\dot B^{s-1}_{p,1}+\dot B^{s-2}_{p,1})},\\[2ex]
\|\dT\nabla V\|_{\dot B^{s-2}_{p,1}}&\lesssim&\|\nabla V\|_{\dot B^{\f32}_{2,1}}\|\dT\|_{\dot B^{s-2}_{p,1}},\\[2ex]
\|\Theta^\eps\nabla\dV\|_{\tilde L^2_T(\dot B^{s-2}_{p,1})}&\lesssim&\|\Theta^\eps\|_{\tilde L^\infty_T(\dot B^{\f12}_{2,1})}
\|\nabla\dV\|_{\tilde L^2_T(\dot B^{s-1}_{p,1})},\\[2ex]
\|\Theta^\eps\nabla\cR^\eps\|_{\tilde L^2_T(\dot B^{s-2}_{p,1})}&\lesssim&\|\Theta^\eps\|_{\tilde L^\infty_T(\dot B^{\f12}_{2,1})}
\|\nabla\cR^\eps\|_{\tilde L^2_T(\dot B^{s-1}_{p,1})},\\[2ex]
\|\cQ u^\eps\cdot\nabla\cP u^\eps\|_{L^1_T(\dot B^{s-1}_{p,1})}&\lesssim&
\|\nabla\cP u^\eps\|_{  L^2_T(B^{\f12}_{2,1})}\|\cQ u^\eps\|_{L^2_T(\dot B^{s}_{p,1})},\\[2ex]
\| u^\eps\cdot\nabla\cQ u^\eps\|_{L^1_T(\dot B^{s-1}_{p,1})}&\lesssim&
\| u^\eps\|_{  L^2_T(B^{\f32}_{2,1})}\|\nabla\cQ u^\eps\|_{  L^2_T(\dot B^{s-1}_{p,1})},\end{array}$$
and arguing as in the proof of \eqref{eq:theta7},
$$\begin{array}{lll}\|\frac{\eps a^\eps}{1+\eps a^{\eps}}\cA u^\eps\|_{L^1_T(\dot B^{s-1}_{p,1})}&\lesssim&
\|\eps a^\eps\|_{L^\infty(\dot B^{\f32-\alpha}_{2,1})}\|\nabla^2 u^\eps\|_{L^1_T(\dot B^{\f12}_{2,1})}\\
&\lesssim&\nu^{-1}(\eps\nu)^\alpha\| a^\eps\|_{L^\infty_T(\tilde B^{\f32,-}_{\eps\nu})}\|u^\eps\|_{L^1_T(\dot B^{\f52}_{2,1})}.
\end{array}
$$
Finally, because $\dot B^{-\f12-\alpha}_{2,1}\hookrightarrow \dot B^{s-2}_{p,1},$
$$\begin{array}{lll}
\|\frac{\eps a^\eps}{1+\eps a^{\eps}}a^\eps\nabla(V^\eps\!+\!\cR^\eps) \|_{\tilde L^2_T(\dot B^{s-2}_{p,1})}
&\!\!\!\!\lesssim\!\!\!\!&\|\eps a^\eps\|_{\tilde L^\infty_T(\dot B^{\f32-\alpha}_{2,1})}\|a^\eps\|_{\tilde L^\infty_T(\dot B^{\f12}_{2,1})}
\|\nabla(V^\eps\!+\!\cR^\eps)\|_{\tilde L^2_T(\dot B^{\f12}_{2,1})}\\
&\!\!\!\!\lesssim\!\!\!\!&\nu^{-1}(\eps\nu)^\alpha\| a^\eps\|_{\tilde L^\infty_T(\tilde B^{\f32,-}_{\eps\nu})
}\|a^\eps\|_{\tilde L^\infty_T(\dot B^{\f12}_{2,1})}
\|V^\eps+\cR^\eps\|_{\tilde L^2_T(\dot B^{\f32}_{2,1})}.
\end{array}
$$
Therefore, putting together all those inequalities
and using the  estimates provided by the previous steps
we conclude that
$$
\displaylines{\|\dv\|_{\tilde L^\infty_T(\dot B^{s-1}_{p,1}+\dot B^{s-2}_{p,1})}
+\|\dv\|_{\tilde L^2_T(\dot B^s_{p,1})+L^1_T(\dot B^s_{p,1})}\lesssim
\|\dv_0\|_{\dot B^{s-1}_{p,1}+\dot B^{s-2}_{p,1}}+\int_0^T\|\nabla V\|_{\dot B^{\f32}_{2,1}}\|\dT\|_{\dot B^{s-2}_{p,1}}\,dt\hfill\cr\hfill+C_0^\eps
(\|\dv\|_{L^2_T(\dot B^s_{p,1})+L^1_T(\dot B^{s}_{p,1})}+\|\dv\|_{\tilde L^\infty_T(\dot B^{s-1}_{p,1}+\dot B^{s-2}_{p,1})}
+\|\nabla\dV\|_{\tilde L^2_T(\dot B^{s-1}_{p,1})}\bigr)
+(C_0^\eps)^2(1+\nu^{-1}C_0^\eps)(\eps\nu)^\alpha.}
$$
If $\nu^{-1}C_0^\eps$ is suitably small, we thus deduce that
$$
\|\dv\|_{\tilde L^\infty_T(\dot B^{s-1}_{p,1}+\dot B^{s-2}_{p,1})}
+\|\dv\|_{\tilde L^2_T(\dot B^s_{p,1})+L^1_T(\dot B^s_{p,1})}\leq \beta(\eps)
+K\int_0^T\|\nabla V\|_{\dot B^{\f32}_{2,1}}\|\dT\|_{\dot B^{s-2}_{p,1}}\,dt
$$
with $\beta(\eps):= \|\dv_0\|_{\dot B^{s-1}_{p,1}+\dot B^{s-2}_{p,1}}
+C_0^\eps(\|\nabla\dV\|_{\tilde L^2_T(\dot B^{s-1}_{p,1})}+C_0^\eps\eps^\alpha).$
\medbreak
Therefore, plugging \eqref{eq:dT} in the above integral, and using Gronwall lemma, we get
\begin{align}\label{eq:final1}
\|\dT\|_{\tilde L^\infty_T(\dot B^{s-2}_{p,1})} \leq
 \bigg(\|\dT_0\|_{\dot B^{s-2}_{p,1}}+(1+T^\f12)\|\Theta_0\|_{\dot B^{\f12}_{2,1}}\big(\|\dv_0\|_{\dot B^{s-1}_{p,1}+\dot B^{s-2}_{p,1}}\hspace{1cm}
 \\\hspace{1cm} +C_0^\eps(\|\nabla\dV\|_{\tilde L^2_T(\dot B^{s-1}_{p,1})}+C_0^\eps\eps^{\f3p-s})\big) +(C_0^\eps)^2(T^{\f12}\eps^{\f3p-s}+\eps T^{ \f3{2p}-\f{s}2})\bigg)\qquad \nonumber\\\hspace{3cm}\times\exp\biggl(K\|\Theta_0\|_{\dot B^{\f12}_{2,1}}(1+T^{\f12})\|\nabla V\|_{L^1_T(\dot B^{\f32}_{2,1})}\biggr),\nonumber
\end{align}
and
\begin{align}  \label{eq:final2}
\|\dv\|_{\tilde L^\infty_T(\dot B^{s-1}_{p,1}+\dot B^{s-2}_{p,1})}
+\|\dv\|_{\tilde L^2_T(\dot B^s_{p,1})+L^1_T(\dot B^s_{p,1})}\hspace{5cm}\\\leq \|\dv_0\|_{\dot B^{s-1}_{p,1}+\dot B^{s-2}_{p,1}}
  +C_0^\eps(\|\nabla\dV\|_{\tilde L^2_T(\dot B^{s-1}_{p,1})}+C_0^\eps\eps^{\f3p-s})\nonumber\\
  +K \bigg(\|\dT_0\|_{\dot B^{s-2}_{p,1}}+(1+T^\f12)\|\Theta_0\|_{\dot B^{\f12}_{2,1}}\big(\|\dv_0\|_{\dot B^{s-1}_{p,1}+\dot B^{s-2}_{p,1}}
\nonumber \\+C_0^\eps(\|\nabla\dV\|_{\tilde L^2_T(\dot B^{s-1}_{p,1})}+C_0^\eps\eps^{\f3p-s})\big) +(C_0^\eps)^2(T^{\f12}\eps^{\f3p-s}+\eps T^{ \f3{2p}-\f{s}2})\bigg)\|\na V\|_{L^1_T(\dot B^{\f32}_{2,1})} \nonumber\\\times\exp\biggl(K\|\Theta_0\|_{\dot B^{\f12}_{2,1}}(1+T^{\f12})\|\nabla V\|_{L^1_T(\dot B^{\f32}_{2,1})}\biggr)\nonumber
 \end{align}
whenever $s\in[-1/2+4/p,3/p]$ and $s>1/2.$
This ensures the convergence of $(\Theta^\eps, \cP u^\eps)$ to $(\Theta, v)$ with an 
explicit rate.\ef


\section{Appendix}
In this Appendix, we  give   some  a priori estimates
involving hybrid Besov spaces.
Let us start with product estimates.

\begin{lem}\label{est-pro} Suppose that  $p\in[2,\infty]$ and $\beta\geq0.$ There exists a constant $C$
such that for all  $\alpha>0,$ we have
\beno
\|fg\|_{\tilde{B}^{s-\beta,-}_{p,\alpha}}&\lesssim& \|f\|_{\tilde{B}^{s,-}_{p,\alpha}}\|g\|_{\dot B^{\f32-\beta}_{2,1}}\quad\hbox{if }\
\beta-1/2<s\leq 3/p,\\
\|fg\|_{\tilde B^{s-\beta,+}_{p,\alpha}}&\lesssim& \|f\|_{\tilde{B}^{s,+}_{p,\alpha}} \|g\|_{\dot B^{\f32-\beta}_{2,1}}\quad\hbox{if }\
\beta-3/2<s\leq 3/p-1.
\eeno
\end{lem}

\proof We may assume that $\alpha=1$ making a change of variables if the case may be.
In order to prove the first inequality, it suffices to notice that for all $\sigma\in\R$ we have
\begin{equation}\label{eq:equiv1}
\|\cdot\|_{\tilde B^{\s,-}_{p,1}}\approx \|\cdot\|_{\dot B^{\s-1}_{p,1}\cap \dot B^{\s}_{p,1}}.
\end{equation}
Now, it is well known (see e.g. \cite{BCD}) that the usual product maps $\dot B^\s_{p,1}\times\dot B^{\f32-\beta}_{2,1}$
in $\dot B^{\s-\beta}_{p,1}$ whenever $\beta-3/2<\s\leq3/p$ and $\beta\geq0.$ Therefore
$$
\|fg\|_{\dot B^{s-\beta-1}_{p,1}}\lesssim \|f\|_{\dot B^{s-1}_{p,1}}\|g\|_{\dot B^{\f32-\beta}_{2,1}}\quad\hbox{and}\quad
\|fg\|_{\dot B^{s-\beta}_{p,1}}\lesssim \|f\|_{\dot B^{s}_{p,1}}\|g\|_{\dot B^{\f32-\beta}_{2,1}}.
$$
This implies the first inequality.
\medbreak
Proving the second inequality is rather similar: now we use the fact that
\begin{equation}\label{eq:equiv2}
\|\cdot\|_{\tilde B^{s,+}_{p,1}}\approx \|\cdot\|_{\dot B^{s}_{p,1}+\dot B^{s+1}_{p,1}}.
\end{equation}
Decomposing $f$ into low and high frequencies according to \eqref{eq:decompo}, we have
$$
fg=f^\ell g+f^h g.
$$
Now, the aforementioned product law ensures that
$$
\|f^\ell g\|_{\dot B^{s+1-\beta}_{p,1}}\lesssim \|f^\ell\|_{\dot B^{s+1}_{p,1}}\|g\|_{\dot B^{\f32-\beta}_{2,1}}\quad
\hbox{and}\quad
\|f^h g\|_{\dot B^{s-\beta}_{p,1}}\lesssim \|f^h\|_{\dot B^{s}_{p,1}}\|g\|_{\dot B^{\f32-\beta}_{2,1}}.
$$
So taking advantage of \eqref{eq:equiv2} completes the proof of the second inequality.
\ef
The following Strichartz estimates for the acoustic wave equation  are the
key to the proof of convergence.
\begin{prop}
\label{p:strichartz}
Let $(q,\cQ u)$ (with $\curl\cQ u=0$) satisfy the 3D acoustic wave equation
$$
\left\{\begin{array}{l} \pa_tq+\sqrt2\,\div\cQ u=F,\\[1ex]
\pa_t\cQ u+\sqrt2\,\nabla q=G.
\end{array}\right.
$$
Then for any $\alpha>0,$ $s\in\R$ and $p\in[2,\infty)$ the following
estimates hold true
$$\begin{array}{lll}
\|(q,\cQ u)\|_{\tilde L^{\f{2p}{p-2}}(\dot B_{p,1}^{s+{\frac2p}-1})}
&\!\!\!\!\leq\!\!\!\!& C\Bigl(\|(q_0,\cQ u_0)\|_{\dot B^s_{2,1}}+\|(F,G)\|_{L^1(\dot B^s_{2,1})}\Bigr),\\[1ex]
\|(q,\cQ u)\|_{\tilde L^{\f{2p}{p-2}}(\tilde B_{p,\alpha}^{s+{\frac2p}-1,\pm})}
&\!\!\!\!\leq\!\!\!\!& C\Bigl(\|(q_0,\cQ u_0)\|_{\tilde B^{s,\pm}_{2,\alpha}}+\|(F,G)\|_{L^1(\tilde B^{s,\pm}_{2,\alpha})}\Bigr).
\end{array}
$$
\end{prop}
\proof The first inequality has been proved in \cite{D3}.
In order to prove the second one, one just has  to decompose
$(q,\cQ u)$ into low and high frequencies, that is $(q,\cQ u)=(q^\ell,\cQ u^\ell)+(q^h,\cQ u^h)$
and apply the first inequality with $s\pm1$ (resp. $s$) to $(q^\ell,\cQ u^\ell)$
(resp. $(q^h,\cQ u^h)$). \ef
Let us finally state
maximal regularity estimates for the heat equation, in hybrid Besov spaces.
\begin{prop}\label{p:heatestimates}
Let $u$ be a solution to the heat equation
$$
\left\{\begin{array}{l} \pa_tu-\Delta u=f,\\
u|_{t=0}=u_0.\end{array}\right.
$$
Then we have the following estimates
for any $\sigma\in\R,$ $\alpha>0,$ $p\in[1,\infty]$ and $q\geq r$:
$$\begin{array}{lll}
\|u\|_{\wt L_T^q(\dot B^{\s+\f2q}_{p,1})}&\lesssim& \|u_0\|_{\dot B^\s_{p,1}}+\|f\|_{\wt L_T^r(\dot B^{\s+\f2r-2}_{p,1})},\\[1ex]
\|u\|_{\wt L_T^q(\wt B^{\s+\f2q,\pm}_{p,\alpha})}&\lesssim& \|u_0\|_{\wt B^{\s,\pm}_{p,\alpha}}
+\|f\|_{\wt L_T^r(\wt B^{\s+\f2r-2,\pm}_{p,\alpha})}.
\end{array}
$$
\end{prop}
\proof
The first inequality is classical (see e.g. \cite{BCD}, Chap. 3).
The second inequality may be obtained from the first one after decomposing $u,$ $u_0$ and $f$
into low and high frequencies.
\ef

\subsubsection*{Acknowledgments.}  The work was  initiated at LAMA of Universit\'e Paris-Est Cr\'eteil
 while  the second author  was on a post-doctoral position  supported by the   CNRS at Cr\'eteil.
 The second author is also supported by NSF of China under Grant 11001149.

\end{document}